\documentclass[a4paper,12pt]{amsart}
\usepackage[utf8]{inputenc}

\usepackage{amsmath,amssymb,amsthm}
\usepackage{bbm}
\usepackage{bm}
\usepackage{xspace}
\usepackage{csquotes}
\usepackage{standalone}

\usepackage[x11names,table]{xcolor}
\usepackage[naturalnames=false]{hyperref}
\usepackage{graphicx}
\graphicspath{ {img/} }
\usepackage{tikz}
\usetikzlibrary{calc}
\usetikzlibrary{cd}
\usepackage{csquotes}
\usepackage{booktabs}
\hypersetup{
	colorlinks = true,          
	urlcolor = DeepSkyBlue3, linkcolor = Chartreuse4, citecolor = DarkOrange2
}
\usepackage{fullpage}
\RequirePackage[backend=bibtex,isbn=false,url=false,maxbibnames=5,bibencoding=utf8]{biblatex}
\renewbibmacro{in:}{, }
\DeclareFieldFormat[misc]{title}{\mkbibquote{#1}}
\DeclareFieldFormat[article]{volume}{\mkbibbold{#1}}
\DeclareFieldFormat{url}{\textsc{url}: \href{#1}{#1}}
\DeclareFieldFormat{doi}{\textsc{doi}: \href{http://dx.doi.org/#1}{#1}}
\DeclareFieldFormat{eprint:arxiv}{arXiv:\href{https://arxiv.org/abs/#1}{#1}}
\usepackage{mathbbol}

\newcommand\CC{{\mathbb C}}

\newcommand\FF{{\mathbb F}}

\newcommand\KK{{\mathbb K}}

\newcommand\RR{{\mathbb R}}

\newcommand\TT{{\mathbb T}}
\newcommand\ZZ{{\mathbb Z}}

\newcommand\cH{{\mathcal H}}

\DeclareMathSymbol{\0}{\mathord}{bbold}{`0}
\DeclareMathSymbol{\1}{\mathord}{bbold}{`1}

\newcommand\SetOf[2]{\left\{#1 \mid #2\right\}}
\newcommand\smallSetOf[2]{\{#1 \mid #2\}}
\newcommand\bigSetOf[2]{\bigl\{#1 \bigm| #2\bigr\}}

\newcommand\torus[1]{\RR^{#1}/\RR\1} 

\newcommand\hahnseries[2]{{#1}\{\hskip-.25em\{{#2}^\RR\}\hskip-.25em\}}

\DeclareMathOperator\dist{dist}
\DeclareMathOperator\conv{conv}

\DeclareMathOperator\val{val}
\DeclareMathOperator\tdet{tdet}

\DeclareMathOperator\SL{SL}

\DeclareMathOperator{\VR}{VR} 
\DeclareMathOperator{\VD}{VD} 
\DeclareMathOperator{\PR}{PR} 
\DeclareMathOperator{\PD}{PD} 
\DeclareMathOperator{\bisector}{bis}
\DeclareMathOperator{\Del}{Del} 

\newtheorem{theorem}{Theorem}
\newtheorem{lemma}[theorem]{Lemma}
\newtheorem{proposition}[theorem]{Proposition}
\newtheorem{corollary}[theorem]{Corollary}

\theoremstyle{remark}

\newtheorem{example}[theorem]{Example}
\newtheorem{definition}[theorem]{Definition}
\newtheorem{remark}[theorem]{Remark}

\definecolor{tolorange}{RGB}{238,119,51}  
\definecolor{tolblue}{RGB}{0,119,187}     
\definecolor{tolgreen}{RGB}{0,153,136}    
\definecolor{tolpurple}{RGB}{136,34,85}   
\definecolor{ibmyellow}{RGB}{255, 176, 0} 
\definecolor{darkgray}{RGB}{64,64,64} 

\tikzstyle{blackdot} = [circle,draw=black,fill=black,scale=0.5]

\usepackage{expl3}
\ExplSyntaxOn
\newcommand \breakDOI[1]
 {
  \tl_map_inline:nn { #1 } { \href{http://dx.doi.org/#1}{##1} \penalty0 \scan_stop: }
 }
\ExplSyntaxOff

\DeclareFieldFormat{doi}{%
  \mkbibacro{DOI}\addcolon\space
  {\breakDOI{#1}}}

\title{Asymmetric tropical distances and power diagrams}

\author{Andrei Com\u{a}neci \and Michael Joswig}

\address[Andrei Com\u{a}neci]{
	Technische Universität Berlin,
	Chair of Discrete Mathematics/Geometry \\
	\texttt{comaneci@math.tu-berlin.de}	
}

\address[Michael Joswig]{
	Technische Universität Berlin,
	Chair of Discrete Mathematics/Geometry \\
	Max-Planck Institute for Mathematics in the Sciences, Leipzig \\
	\texttt{joswig@math.tu-berlin.de}
}

\thanks{Support by the Deutsche Forschungsgemeinschaft (DFG, German Research Foundation) under Germany's Excellence Strategy -- The Berlin Mathematics Research Center MATH$^+$ (EXC-2046/1, project ID 390685689) gratefully acknowledged.
  M.~Joswig has further been supported by \enquote{Symbolic Tools in Mathematics and their Application} (TRR 195, project-ID 286237555).
  A.~Com\u{a}neci was supported by \enquote{Facets of Complexity} (GRK 2434, project-ID 385256563).}


\begin{document}

\begin{abstract}
  We investigate the Voronoi diagrams with respect to an asymmetric tropical distance function, also for infinite point sets.
  These turn out to be much better behaved than the tropical Voronoi diagrams arising from the standard tropical distance, which is symmetric.
  In particular, we show that the asymmetric tropical Voronoi diagrams may be seen as tropicalizations of power diagrams over fields of real Puiseux series.
  Our results are then applied to rational lattices and Laurent monomial modules.
\end{abstract}

\maketitle

\section{Introduction}
\noindent
The asymmetric tropical distance function is the polyhedral norm on the quotient vector space $\torus{n}$ induced by the regular simplex $\conv\{e_1,e_2,\dots,e_n\}$.
In this way, this is a special case of a polyhedral norm with respect to a polytope which is not centrally symmetric; cf.\ \cite[Sect.~7.2]{AurenhammerKleinLee:2013} and \cite[Sect.~4]{MartiniSwanepoel:2004}.
Such a not necessarily symmetric norm is also called a \enquote{polyhedral gauge}.
The asymmetric tropical distance function and the resulting Voronoi diagrams were studied by Amini and Manjunath \cite{RRlattice} and Manjunath \cite{LapLatticeGraph} in the context of tropical versions of the Riemann--Roch theorem for algebraic curves.
Recently, we analyzed the Fermat--Weber problem for the same distance function \cite{ComaneciJoswig:2205.00036}.
For general introductions to tropical geometry see \cite{Tropical+Book} and \cite{ETC}.

One motivation for this work is a recent trend to exploit metric properties of tropical linear spaces and more general tropical varieties for applications in optimization and data science.
Usually, these results employ the tropical distance function $\dist(a,b) = \max (a_i-b_i) - \min (a_j-b_j)$, which is symmetric.
That symmetry seems to suggest that this is the natural distance function to work with.
Moreover, the tropical distance function has a history as \enquote{Hilbert's projective metric}, which was investigated, e.g., by Cohen, Gaubert and Quadrat \cite{CohenGaubertQuadrat04}.
Yet here we gather further evidence that its asymmetric sibling is actually better behaved, geometrically and algorithmically.
This is in line with the observation that the asymmetric tropical Fermat--Weber problem \cite{ComaneciJoswig:2205.00036} is more benign than its symmetric counterpart, which was considered by Lin and Yoshida \cite{Lin-Yoshida:2018}.

Our contributions include the following.
First, we extend the study of the polyhedral geometry of the Voronoi diagrams with respect to the asymmetric tropical distance, a topic which arose in \cite{RRlattice}.
We explicitly admit point sets which are infinite.
For instance, the bisectors turn out to be tropically convex and thus contractible.
This is very different from the symmetric case, where topologically nontrivial bisectors do occur \cite[Example~3]{TropBis}.
We prove that locally the asymmetric tropical Voronoi regions of a discrete set of sites behave like (possibly unbounded) tropical polyhedra; yet globally this is not true in general.
This leads us to defining a new notion of \emph{super-discreteness}, for which we can show that the asymmetric tropical Voronoi regions do form tropical polyhedra (Theorem~\ref{th:polySuperDisc}).
This includes finite sets and rational lattices as special cases.

As a second main result, we prove that the asymmetric tropical Voronoi diagrams arise as tropicalizations of power diagrams over fields of real Puiseux series (Theorem~\ref{th:PuiseuxLift}).
In this way we further explore the connection between tropical convexity and ordinary polyhedral geometry over ordered fields; see \cite[\textsection 2]{DevelinYu07} and \cite[\textsection 5.2]{ETC}.
The relationship between tropical Voronoi diagrams and ordinary power diagrams is also reminiscent of the classical construction of Euclidean Voronoi diagrams through projecting a convex polyhedron whose facets are tangent to the standard paraboloid.
Power diagrams generalize Voronoi diagrams much like regular subdivisions generalize Delone subdivisions of point sets in Euclidean space.
A similar construction for Voronoi diagrams with respect to the symmetric tropical distance is unknown and seems unlikely to exist. 
Edelsbrunner and Seidel investigated Voronoi diagrams for very general distance functions \cite{Edelsbrunner+Seidel:1986}.
Their construction differs from the one developed by Amini and Manjunath \cite{RRlattice}, which we adopt here.
Yet, it turns out that for discrete sets in general position both notions agree (Theorem~\ref{thm:V-cells}).

Finally, we define \emph{asymmetric tropical Delone complexes}.
These are abstract simplicial complexes, which may be somewhat unwieldy in general.
For sites in general position, however, this corresponds to an ordinary Delone triangulation over Puiseux series.
This is interesting because Scarf complexes of Laurent monomial modules arise as special cases (Corollary~\ref{cor:scarf}).
These occur as supports of resolutions in commutative algebra \cite{CombCommAlg}.

Delone complexes are named after the Russian mathematician Boris Nikolayevich Delone (1880--1980), whose name is sometimes written \enquote{Delaunay}.
Delone himself used both spellings in his articles.


\section{Directed distances and polyhedral norms}
\noindent
We start out by fixing our notation.
Consider a polytope $K \subset \RR^n$ with the origin in its interior; the existence of an interior point entails that $\dim K=n$.
For $a,b\in\RR^n$ let $\dist_K(a,b)$ be the unique real number $\alpha>0$ such that $b-a \in \partial (\alpha K)$, where $\alpha K$ is $K$ scaled by $\alpha$.
We call this the \emph{(polyhedral) directed distance} from $a$ to $b$ with respect to $K$.
The function $d_K$ satisfies the triangle inequality, and it is translation invariant and homogeneous with respect to scaling by positive reals.
If additionally $K$ is centrally symmetric, i.e., $K=-K$, then $d_K(a,b)=d_K(b,a)$.
Consequently, in this case $d_K$ is a metric, and $d_K(\cdot,\0)$ is a norm on $\RR^n$.
Norms which arise in this way are called \emph{polyhedral}; see \cite[Sect.~7.2]{AurenhammerKleinLee:2013} and \cite[Sect.~4]{MartiniSwanepoel:2004}.
We write $\0$ for the all zeros and $\1$ for the all ones vectors (of length $n$), respectively.
The \emph{asymmetric tropical distance} in $\RR^n$ is given by
\begin{equation}\label{eq:dtriangle}
  d_\triangle(a,b) \ = \ \sum_{i\in[n]} (b_i - a_i) - n \min_{i\in[n]}(b_i-a_i) \ = \ \sum_{i\in[n]} (b_i -a_i) + n \max_{i\in[n]}(a_i-b_i) \enspace ,
\end{equation}
where $a, b\in\RR^n$.
Since $d_\triangle(a',b')=d_\triangle(a,b)$ for $a-a'\in\RR\1$ and $b-b'\in\RR\1$, this induces the directed distance with respect to the standard simplex $\triangle=\conv\{e_1,\dots,e_n\}+\RR\1$ in the $(n{-}1)$-dimensional quotient vector space $\torus{n}$, which is the \emph{tropical projective torus}.
We do not distinguish between these two functions, and call the latter directed distance function also the \emph{asymmetric tropical distance} in $\torus{n}$.

The \emph{(symmetric) tropical distance} between $a,b\in\RR^n$ (or $\torus{n}$) is more common.
It is defined as
\[
  \dist(a,b) \ = \ \max_{i\in[n]} \left(a_i-b_i\right) - \min_{j\in[n]} \left(a_j-b_j\right) \ = \ \max_{i,j\in[n]}(a_i-b_i-a_j+b_j) \enspace ;
\]
cf.\ \cite[\textsection 5.3]{ETC}.
The symmetric and asymmetric tropical distances are related by
\[
  \begin{aligned}
    \dist(a,b) \ &= \ \tfrac{1}{n}\left(d_{\triangle}(a,b) + d_\triangle(b,a)\right) \quad \text{and}\\
    \dist(a,b) \ &\leq \ d_\triangle(a, b) \ \leq \ (n-1)\dist(a,b) \enspace ,
  \end{aligned}
\]
where the two inequalities are both tight.

\section{Asymmetric tropical Voronoi regions}
\label{sec:regions}
\noindent
In this section we will investigate Voronoi diagrams with respect to the asymmetric tropical distance function.
Our results extend the work of Amini and Manjunath \cite[\textsection 4.2]{RRlattice}, who study these Voronoi diagrams for points located in a sublattice of a root lattice of type $A$.
In a way, here we pick up the suggestion in \cite[Remark 4.9]{RRlattice} to make the connection to tropical convexity.

It will be convenient to work with special coordinates in $\torus{n}$.
For this, we consider the tropical hypersurface
\[
  \cH \ := \SetOf{x\in\RR^n}{ x_1+x_2+\cdots+x_n=0 } \enspace ,
\]
which is an ordinary linear hyperplane in $\RR^n$.
The hyperplane $\cH$ occurs as \enquote{$H_0$} in \cite{RRlattice}.
Observe that $-\cH=\cH$ is centrally symmetric.
This makes $\cH$ a tropical hyperplane with respect to both choices of the tropical addition, $\max$ and $\min$; cf.\ \cite[\textsection 1.3]{ETC}.
Moreover, each point $x+\RR\1\in\torus{n}$ has a unique representative with $\sum x_i=0$.
This gives a linear isomorphism $\torus{n}\cong\cH$ that will be used throughout the paper.
In fact, we will state most of our results using $\cH$ instead of $\torus{n}$ to emphasize this identification.

For our arithmetic we consider the \emph{$\max$-tropical semiring} $\TT=(\RR\cup\{-\infty\},\oplus,\odot)$, where $\oplus=\max$ is the tropical addition, and $\odot=+$ is the tropical multiplication.
The additive neutral element is $-\infty$, and $0$ is neutral with respect to the tropical multiplication.
A (homogeneous) \emph{tropical linear inequality} looks like
\[
  \bigoplus_{i\in I} a_i\odot x_i \ \leq \ \bigoplus_{i\in J} b_j\odot x_j \enspace ,
\]
where $a_i,b_j\in\RR$, and $I,J$ are disjoint nonempty subsets of $[n]$.
The set of solutions is a \emph{tropical halfspace} in $\torus{n}$ (or $\cH$).
A \emph{tropical polyhedron} is the intersection of finitely many tropical halfspaces; see\ \cite[\textsection 7.2]{ETC}.
Replacing $\max$ by $\min$ and $-\infty$ by $\infty$ we arrive at the \emph{$\min$-tropical semiring}, which is isomorphic to $\TT$ as a semiring.
Here we stick to the $\max$-convention throughout.
Nonetheless, we often stress this choice as this is a consequence of how we defined $d_\triangle$ in \eqref{eq:dtriangle}.
A set $C\subset\RR^n$ is \emph{tropically convex} if for all $x,y\in C$ and $\lambda,\mu\in\RR$ we have $\lambda{\odot}x \oplus \mu{\odot}y\in C$.
Each nonempty tropically convex set contains $\RR\1$, whence we may study tropical convexity in $\torus{n}$; cf.\ \cite[\textsection 5.2]{ETC}.

Let $S\subset\torus{n}$ be a nonempty discrete set of points, which is possibly infinite.
Following common practice in computational geometry, the points in $S$ will be called the \emph{sites}.
Throughout, we will measure distances via the asymmetric tropical distance function $d_\triangle$.
The \emph{(asymmetric tropical) Voronoi region} of a site $a\in S$ is the set
\[
  \VR_S(a) \ := \ \bigSetOf{x\in\cH}{d_\triangle(x,a) \leq d_\triangle(x,b) \text{ for all } b\in S} \enspace .
\]
For two distinct sites $a,b\in\cH$ we abbreviate $h(a,b)=\VR_{\{a,b\}}(a)$.
This notation yields $\VR_S(a)=\bigcap_{b\in S\setminus\{a\}} h(a,b)$.
The analysis of asymmetric tropical Voronoi regions starts with the following basic observation, which is similar to \cite[Lemma~1]{Gaubert+Katz:2011}.
\begin{proposition}\label{prop:VR-pair}
  For two distinct points $a,b\in\cH$ the Voronoi region $h(a,b)$ is a $\max$-tropical halfspace.
\end{proposition}
\begin{proof}
  The inequality $d_\triangle(x,a) \leq d_\triangle(x,b)$ translates into
  \[
     n \max_{i\in[n]}(x_i-a_i) \ \leq \  n \max_{i\in[n]}(x_i-b_i) \enspace.
  \]
  Then the above inequality is equivalent to
  \begin{equation}\label{eq:VR-pair:ineq}
    \bigoplus_{i\in[n]} (-a_i) \odot x_i \ \leq \ \bigoplus_{i\in[n]} (-b_i) \odot x_i \enspace .
  \end{equation}
  Since further $a\neq b$, the difference $a-b$ must have positive as well as negative coordinates.
  Consequently, the set $I:=\smallSetOf{i\in[n]}{a_i<b_i}$ is a nonempty and proper subset of $[n]$.
  We will show that \eqref{eq:VR-pair:ineq} is equivalent to
  \begin{equation}\label{eq:VR-pair:halfspace}
    \bigoplus_{i\in I} (-a_i) \odot x_i \ \leq \ \bigoplus_{i\not\in I} (-b_i) \odot x_i \enspace .
  \end{equation}
  Suppose that $x\in\cH$ satisfies \eqref{eq:VR-pair:ineq}.
  Let $k,\ell\in[n]$ be indices with $\max (-a_i + x_i)=-a_k + x_k \leq -b_\ell + x_\ell = \max (-b_i + x_i)$.
  We distinguish two cases.
  Either $k\in I$, whence $-a_k>-b_k$ and thus $\ell\not\in I$; so $x$ satisfies \eqref{eq:VR-pair:halfspace}.
  Or $k\not\in I$, whence $-a_k\leq -b_k$; in this case $x$ trivially satisfies \eqref{eq:VR-pair:halfspace}.
  This argument can be reverted, which proves the reverse implication.
  The homogeneous max-tropical linear inequality \eqref{eq:VR-pair:halfspace} describes a max-tropical linear halfspace.
\end{proof}
The above result has a direct consequence, which sharpens \cite[Lemma~4.4]{RRlattice}.
\begin{corollary}
  For an arbitrary finite set of sites $S\subset\cH$ and $a\in S$ the Voronoi region $\VR_S(a)$ is a max-tropical tropical polyhedron.
\end{corollary}
\begin{proof}
  We have $\VR_S(a) = \bigcap_{b\in S\setminus\{a\}} h(a,b)$, and thus the claim follows from the previous result.
\end{proof}

We define the \emph{(asymmetric tropical) Voronoi diagram} of $S$, denoted $\VD(S)$, as the intersection poset generated by the Voronoi regions.
The intersections of nonempty families of Voronoi regions are the \emph{(asymmetric tropical) Voronoi cells}.
These form the elements of $\VD(S)$, and they are partially ordered by inclusion.
The \emph{(asymmetric tropical) bisector} of $S$ is the set
\[
  \bisector(S) \ = \ \bigSetOf{x\in\cH}{d_\triangle(x,a) = d_\triangle(x,b) \text{ for all } a,b\in S} \ = \ \bigcap_{a\in S} \VR_S(a) \enspace .
\]
We have the following basic topological information about bisectors and Voronoi cells.
\begin{corollary}
  Any bisector $\bisector(S)$ and any cell of $\VD(S)$ is $\max$-tropically convex and hence contractible or empty.
\end{corollary}
\begin{proof}
  The boundary plane of a tropical halfspace is tropically convex \cite[Observation~7.5]{ETC}.
  The intersection of tropically convex sets is tropically convex.
  Nonempty tropically convex sets are contractible \cite[Proposition~5.22]{ETC}.
\end{proof}
By \cite[Theorem~7.11]{ETC} the closure of any bisector or Voronoi cell in the tropical projective space is a tropical polytope.
This is in stark contrast with the situation in symmetric tropical Voronoi diagrams, which allow for topologically nontrivial bisectors; cf.\ \cite[Example~3]{TropBis}.

\begin{figure}[th]
  \centering
  \begin{tikzpicture}[semithick,x={(0.866,-0.5)},y={(0,1)},z={(-0.866,-0.5)}]

		\definecolor{tolorange}{RGB}{238,119,51}  
		\definecolor{tolblue}{RGB}{0,119,187}     
		\definecolor{tolgreen}{RGB}{0,153,136}    
		\definecolor{tolpurple}{RGB}{136,34,85}   
		\definecolor{ibmyellow}{RGB}{255, 176, 0} 



		\fill[color=tolorange, opacity=0.25] (-12,-1) -- (-5,-1) -- (-5,-2) -- (-2.5,-2) -- (-2.5,-3) -- (-1.666,-3) -- (-1.666,-4) -- (-1.25,-4) -- (-1.25,-5) -- (-1,-5) -- (-1,-6) -- (-0.833,-6) -- (-0.833,-7) -- (-0.714,-7) -- (-0.714,-8) -- (-0.625,-8) -- (-0.625,-9) -- (-0.555,-9) -- (-0.555,-10) -- (-0.5,-10) -- (-0.5,-12) -- (3,-10.5) -- (3,8) -- (-12,2) -- cycle;
		

		\node at (0,0) [circle,fill,inner sep=3pt,color=tolpurple,draw = white, semithick] {};
		
	    \node at (-11,-7) [circle,fill,inner sep=3pt,color=tolpurple,draw = white, semithick] {};
	    \node at (-7,-6.5) [circle,fill,inner sep=3pt,color=tolpurple,draw = white, semithick] {};
	    \node at (-6.333,-7.666) [circle,fill,inner sep=3pt,color=tolpurple,draw = white, semithick] {};
	    \node at (-6.5,-9.25) [circle,fill,inner sep=3pt,color=tolpurple,draw = white, semithick] {};
	    \node at (-7,-11) [circle,fill,inner sep=3pt,color=tolpurple,draw = white, semithick] {};
	    \node at (-7.666,-12.833) [circle,fill,inner sep=3pt,color=tolpurple,draw = white, semithick] {};
	    \node at (-8.428,-14.714) [circle,fill,inner sep=3pt,color=tolpurple,draw = white, semithick] {};
	    

\end{tikzpicture}
  \caption{The non-polyhedral Voronoi region $\VR_S(\0)$ from Example~\ref{ex:nonpolyh}, for $a=5$.}
  \label{fig:discrete-nonpolyhedral}
\end{figure}
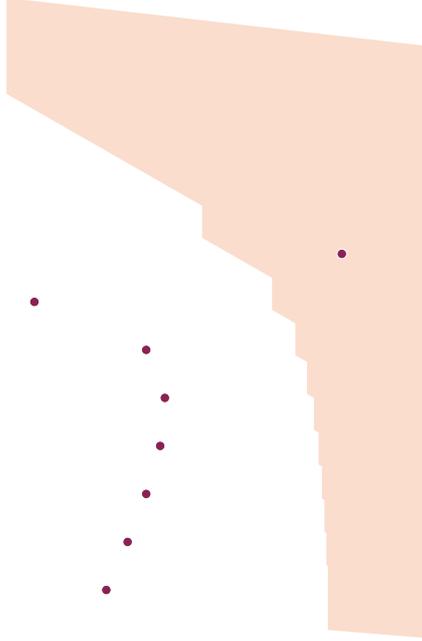

Euclidean Voronoi regions for a general discrete set need not be polyhedral, but they are always \emph{quasi-polyhedra}, i.e., their intersections with polytopes yield polytopes \cite[Proposition~32.1]{Gruber:2007}.
The next example shows that the situation is similar in the tropics.
Comprehensive discussions on Euclidean Voronoi diagrams can be found in \cite{Voigt:2008} and \cite{AurenhammerKleinLee:2013}.
\begin{example} \label{ex:nonpolyh}
  For a fixed positive real number $a$, consider the infinite set
  \[ S \ = \ \{\0\}\cup\SetOf{\left(n+\tfrac{a}{n},-\tfrac{a}{n},-n\right)}{n\in\ZZ_{>0}} \enspace, \]
  which is discrete.
  Then the Voronoi cell $\VR_S(\0)$ is defined by the inequalities $x_1\leq\max(x_2+\tfrac{a}{n},x_3+n)$ for all positive integers $n$.
  None of these are redundant.
  Therefore, $\VR_S(\0)$ is not a tropical polyhedron; see Figure~\ref{fig:discrete-nonpolyhedral}.
\end{example}

A point set $S\subset\cH$ is in \emph{general position} if for any pair of distinct sites $a,b\in S$ we have $a_i\neq b_i$ for all $i\in[n]$.
The following observation is similar to \cite[Proposition~2]{TropBis}, which is about the symmetric tropical distance function.
\begin{lemma}\label{lem:bisector-generic}
  For two distinct points $a,b\in\cH$ in general position the two point bisector $\bisector(a,b)$ is the boundary plane of a $\max$-tropical halfspace.
  In particular, it is an ordinary polyhedral complex of codimension one.
\end{lemma}
\begin{proof}
  By Proposition~\ref{prop:VR-pair} the bisector $\bisector(a,b)$ is the intersection of two opposite tropical halfspaces.
  Since $a$ and $b$ are generic no sector of the corresponding tropical hyperplane is contained in the intersection, and this is the claim.
\end{proof}

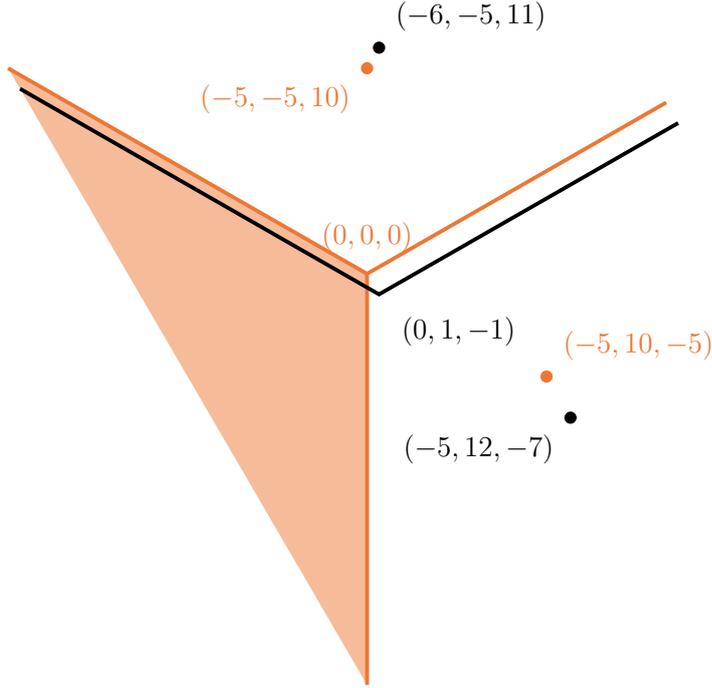
\begin{figure}[th]
  \centering
  \begin{tikzpicture}[ultra thick,x={(0.866,-0.5)},y={(0,1)},z={(-0.866,-0.5)}]
		\definecolor{tolorange}{RGB}{238,119,51} 
		\definecolor{tolblue}{RGB}{0,119,187}    
		\definecolor{tolgreen}{RGB}{0,153,136}   
		\definecolor{tolpurple}{RGB}{136,34,85}  
		\definecolor{TUrot}{RGB}{197, 14, 31}    



		\fill[color=tolorange, opacity=0.5] (-6,0) -- (0,0) -- (0,-6) -- cycle;

			\node at (3,0) [circle,fill,inner sep=2pt,color=tolorange,draw = white, semithick,label={[tolorange]above right:$(-5,10,-5)$}] {};
	    \node at (0,3) [circle,fill,inner sep=2pt,color=tolorange,draw = white, semithick,label={[tolorange]below left:$(-5,-5,10)$}] {};
	    
	    \node at (0.2,3.4) [circle,fill,inner sep=2pt,color=black,draw = white, semithick,label={above right:$(-6,-5,11)$}] {};
	    \node at (3.4,-0.4) [circle,fill,inner sep=2pt,color=black,draw = white, semithick,label={below left:$(-5,12,-7)$}] {};
	    
	    \node at (0,0) [label={[tolorange]above:$(0,0,0)$}] {};
	    \node at (0.2,-0.2) [label={below right:$(0,1,-1)$}] {};
	    
	    \draw[color=tolorange] (0,0) -- (5,5);
	    \draw[color=tolorange] (-6,0) -- (0,0) -- (0,-6);
	    
	    \draw (-5.8,-0.2) -- (0.2,-0.2) -- (5.2,4.8);

\end{tikzpicture}
  \caption{Bisectors of two sites in $\torus{3}$: general position vs.\ non-general position}
  \label{fig:genpos}
\end{figure}

The assumption of general position is needed in Lemma~\ref{lem:bisector-generic}, as the next example shows.
\begin{example}
  Consider $a=(-6,-5,11)$ and $b=(-5,12,-7)$, which are in general position.
  The bisector $\bisector(a,b)$ is the boundary of a $\max$-tropical halfspace with apex $\min(a,b)=(-6,-5,-7)\in\torus{3}$, whose representative in $\cH$ is $(0,1,-1)$.
  However, the points $c=(-5,-5,10)$ and $d=(-5,10,-5)$ are not in general position.
  The bisector $\bisector(c,d)$ contains the entire $\max$-tropical hyperplane with apex $\min(c,d)=(-5,-5,-5)=(0,0,0)$ and the sector $\smallSetOf{x}{x_1\geq\max(x_2,x_3)}$.
  Both cases are illustrated in Figure~\ref{fig:genpos}. 
\end{example}

If $S$ is not in general position, the Voronoi regions might not be pure dimensional; see Example~\ref{exmp:lattice} below.
This happens because the closure of the interior of a Voronoi region might be a proper subset of the region.
However, for $S$ in general position, this situation does not occur.

\begin{lemma}\label{lem:puredim}
  If $S\subset\cH$ is a discrete set in general position, then any Voronoi region $\VR_S(a)$ agrees with the closure of its interior.
\end{lemma}

\begin{proof}
  Let $x\in\VR_S(a)$ and $x^{(t)}=a+t(x-a)$ for $t\in[0,1)$.
  Consider also an arbitrary site $b\in S\setminus\{a\}$.
  Due to the general position of $S$, we have
  \[ \max_{i\in I(a,b)}(x_i-a_i) \ \leq \ \max_{j\in I(b,a)}(x_j-b_j) \enspace ,\]
  where $I(a,b):=\smallSetOf{i\in[n]}{a_i<b_i}$ as $x\in h(a,b)$.
  Then
  \[\begin{split}
      \max_{i\in I(a,b)}\left(x^{(t)}_i-a_i\right) \ &= \ t\max_{i\in I(a,b)}(x_i-a_i)\\
      & \leq \ t\max_{j\in I(b,a)}(x_j-b_j) \ = \ \max_{j\in I(b,a)}t(x_j-b_j)\\
      & < \ \max_{j\in I(b,a)}\left((1-t)(a_j-b_j)+t(x_j-b_j)\right)\ = \ \max_{j\in I(b,a)}\left(x^{(t)}_j-b_j\right) \enspace .
    \end{split}
  \]
  As $S$ is discrete, the interior of $\VR_S(a)$ is the intersection of the interiors of the halfspaces $h(a,b)$ for $b\neq a$.
  So the previous inequality shows that $x^{(t)}$ belongs to the interior of $\VR_S(a)$ for every $t\in[0,1)$.
  The conclusion follows because $\lim_{t\to 1}x^{(t)}=x$.
\end{proof}

\section{Super-discrete sets of sites}
\noindent
We continue the study of a possibly infinite number of sites.
In Example~\ref{ex:nonpolyh} we saw that tropical Voronoi regions do not need to be tropical polyhedra in general.
Nonetheless, tropical Voronoi regions are always locally polyhedral.
Here we will explore the details, and we will develop a notion which forces tropical polyhedrality.
This turns out to be applicable to tropical Voronoi diagrams of lattices.

A subset of $\cH$ is a \emph{tropical quasi-polyhedron} if its intersection with any bounded tropical polyhedron is a tropical polyhedron.

\begin{proposition}
  The tropical Voronoi regions of a discrete set $S\subset\cH$ are tropical quasi-polyhedra.
\end{proposition}

\begin{proof}
  Up to a translation, we can assume that $\0$ is among the sites $S$.
  It suffices to show that $\VR_S(\0)$ intersected with any symmetric tropical ball around $\0$ is a tropical polyhedron.
  For $M\in\RR_{>0}$ let $B_M$ the symmetric tropical ball given by $x_i-x_j\leq M$ for all $i,j\in[n]$.
  Due to the discreteness of $S$ there are only finitely many points of $S$ inside the cube $C_M = [-(n-1)M,(n-1)M]^n$.

  Consider $a\in S$ which lies outside $C_M\cap\cH$.
  Then there exists $i\in [n]$ with $|a_i|>(n-1)M$.
  We claim that there is an index $j\in[n]$ such that $a_j < -M$.
  To see this, assume the contrary.
  We have $\sum_{k\in[n]}a_k=a_i+\sum_{k\neq i}a_k>(n-1)M+(n-1)\cdot(-M)=0$, which contradicts $a\in\cH$.

  Let $x\in B_M$.
  Then $\max_{k\in[n]}(-a_k+x_k)\geq -a_j+x_j>M+x_j\geq\max_{k\in [n]}x_k$.
  This yields $B_M\subseteq h(\0,a)$.
  Therefore, $\VR_S(\0)\cap B_M=\bigcap_{s\in C_M\cap(S\setminus\{\0\})}h(\0,s)\cap B_M$ is a bounded intersection of finitely many tropical halfspaces, i.e., a tropical polyhedron.
\end{proof}

Bounded tropical quasi-polyhedra are tropical polyhedra. This gives us a first criterion to check if a Voronoi region is a tropical polyhedron.

\begin{corollary}
  If $S$ is a discrete set and $\VR_S(s)$ is bounded for some $s$, then $\VR_S(s)$ is a tropical polyhedron.
\end{corollary}

Next we describe a class of (finite or infinite) sets with nice Voronoi regions.
\begin{definition}
  Let $r,R\in\RR\cup\{\infty\}$ such that $0<r<R$.
  A set $S\subset\cH$ is called an \emph{$(r,R)$-system} if $(s+r\triangle)\cap S=\{s\}$ for all $s\in S$ and $(x+R\triangle)\cap S\neq\emptyset$ for all $x\in\cH$.
\end{definition}
For $R=\infty$, we consider $R\triangle=\cH$, so an $(r,\infty)$-system imposes only a uniform lower bound on the distances between sites.
When $R$ is finite, the above definition agrees with \cite[Definition 3.1.4]{Voigt:2008} despite being formally different.
In that case, our version is equivalent because the asymmetric tropical distance is strongly equivalent to the Euclidean distance (i.e., there exist $\alpha,\beta>0$ such that $\alpha d_{L^2}(x,y)\leq d_{\triangle}(x,y)\leq \beta d_{L^2}(x,y)$ for all $x,y\in\cH$, where $d_{L^2}$ is the Euclidean distance).
Occasionally, $(r,R)$-systems are also called \enquote{Delone sets} in the literature; e.g., see \cite[loc.\ cit.]{Voigt:2008}.
Here we prefer the slightly more sterile terminology to avoid a confusion with the Delone complexes discussed in Section~\ref{sec:delaunay} below.
The following generalizes \cite[Lemma 4.6]{RRlattice}.
\begin{lemma}\label{lem:compact}
  For $R<\infty$ the tropical Voronoi regions of $(r,R)$-systems are bounded and thus compact.
\end{lemma}

\begin{proof}
Let $S$ be an $(r,R)$-system in $\cH$ and $s\in S$.
Up to translation, we can assume that $s=\0$.
Select $i\in[n]$ arbitrary and consider the cone $C_i$ given by the equations $x_i \leq 0$ and $x_j > 0$ for all $j\neq i$.
Since $C_i$ is a full-dimensional convex cone, one can find $x\in C_i$ such that $x+R\triangle\subset C_i$.
Then there exists $t\in S\cap(x+R\triangle)$ because $S$ is an $(r,R)$-system.
In particular, $t$ is a site contained in $C_i$.

The equation of the tropical halfspace $h(\0,t)$ is $\max_{j\neq i}x_j\leq x_i-t_i$. From the fact that $\sum_{j\neq i}x_j=-x_i$, we obtain $-\frac{1}{n-1}x_i\leq\max_{j\neq i}x_j\leq x_i-t_i$ and thus $x_i\geq\frac{n-1}{n}t_i$ for every $x\in h(\0,t)$.
This restricts, particularly, to points of $\VR_S(\0)$, which is a subset of $h(\0,t)$.

We have shown that, for an arbitrary $i\in[n]$, the $i$-th coordinate $x_i$ is bounded uniformly from below for every $x\in\VR_S(\0)$.
In other words, there exists $\delta>0$ such that for every $x\in\VR_S(\0)$ and $i\in [n]$ we have $x_i\geq -\delta$.
Using again the property $\sum x_i=0$, we also obtain $x_i\leq(n-1)\delta$ for all $x\in\VR_S(\0)$ and $i\in [n]$.

Summing up, $\VR_S(\0)$ is a subset of the cube $[-\delta,(n-1)\delta]^n$, so it is bounded.
\end{proof}

It is a consequence of Lemma~\ref{lem:compact} that $(r,R)$-systems are necessarily infinite if $R$ is finite.
Thus, these form a quite restricted class of sets, whereas the analysis in \cite{Voigt:2008} admits arbitrary unbounded polyhedral Euclidean Voronoi regions for discrete sets.
We now describe a class of sets whose tropical Voronoi regions are always tropical polyhedra, even when they are unbounded.
To this end we need to introduce some notation.
For a subset $I$ of $[n]$ we consider the projection $\pi_I:\RR^n\rightarrow\RR^{|I|}$ mapping a point $x$ to $(x_i)_{i\in I}$, its entries with coordinates in $I$.

\begin{definition}\label{def:super-discrete}
  We call a set $S\subset\cH$ \emph{super-discrete} if for every $i\in[n]$ the projection $\pi_i(S)$ is a discrete subset of $\RR$. 
\end{definition}

A super-discrete set is, in particular, a discrete subset of $\cH$.
To see this, notice that every cube $[m,M]^n$ intersected with $\cH$ contains finitely many points of a super-discrete set; the final assertion can be shown inductively.
Examples of super-discrete sets include finite sets and rational lattices.
For irrational lattices it may happen that a projection onto one coordinate is dense in $\RR$; see Example~\ref{exmp:irr} below.
The set in Example~\ref{ex:nonpolyh} is discrete but not super-discrete: the projection onto the second coordinate is not a discrete set.

The following examples show that super-discrete sets and $(r,R)$-systems are different.
\begin{example}\label{exmp:ln}
  The sequence of sites $s_i=(\ln(i),-\ln(i))$ for $i\in\ZZ_{>0}$ forms a super-discrete set in $\torus{2}$, but $d_\triangle(s_i,s_{i+1})=2\ln((i+1)/i)$ tends to zero as $i$ goes to infinity.
  So it is not an $(r,R)$-system for any $0<r<R$.
\end{example}

\begin{example}\label{exmp:irr}
  Let $L$ be the lattice generated by $(1,1,-2)$ and $(0,\sqrt{2},-\sqrt{2})$, which is not rational.
  Then $\pi_2(L)=\SetOf{\alpha+\beta\sqrt{2}}{\alpha,\beta\in\ZZ}$ is dense in $\RR$ by Kronecker's density theorem; see \cite[Chapter~XXIII]{HardyWright:2008}.
  Hence, $L$ is not super-discrete; yet it is an $(r,R)$-system.
\end{example}

\begin{remark} \label{rmk:super-discrete}
  By definition, subsets of super-discrete sets are also super-discrete.
  Moreover, any projection $\pi_I(S)$ of a super-discrete set $S$ is super-discrete.
\end{remark}

An element $x$ of a set $G\subseteq\RR^n$ is called \emph{nondominated} if there is no $y\in G\setminus\{x\}$ such that $x\leq y$.
This notion appears in multicriteria optimization \cite{Ehrgott:2005}; see also \cite{MonCon} for a connection to tropical combinatorics.

\begin{lemma} \label{lemma:supdisc-fingen}
Let $G$ be a super-discrete subset of $\RR^n_{\geq 0}$.
Then there are finitely many nondominated points in $G$.
\end{lemma}

\begin{proof}
  For all $i\in[n]$ the sets $\pi_i(G)$ are discrete and thus countable.
  So $\pi_i(G)$ is a countable subset of $\RR_{\geq 0}$, whence it admits a well-ordering.
  Thus, we can pick an order-preserving embedding $\rho_i:\pi_i(G)\rightarrow\ZZ_{>0}$ for every $i\in[n]$.

  Therefore, the map $\rho:G\rightarrow\ZZ_{>0}^n$, $x\mapsto(\rho_1(\pi_1(x)),\dots,\rho_n(\pi_n(x)))$ is injective and order-preserving.
  In particular, we have a bijection between the nondominated points of $G$ and the nondominated points of $\rho(G)$.
  But $\rho(G)$ has finitely many nondominated points by Dickson's lemma \cite[Theorem 2.1.1]{HerzogHibi:2011}.
\end{proof}

Dickson's lemma has a prominent role in commutative algebra.
That it occurs here is not a coincidence; see Section~\ref{sec:delaunay} below.
For now we are content with the following combinatorial result.

\begin{theorem} \label{th:polySuperDisc}
  If $S$ is a super-discrete subset of $\cH$, then all the cells of the Voronoi diagram $\VD(S)$ are tropical polyhedra.
\end{theorem}

\begin{proof}
Let $s\in S$ be arbitrary.
After a translation, we can assume that $s=\0$.
We prove that $\VR_{S}(\0)$ is a tropical polyhedron by showing that there exists a finite subset $T$ of $S\setminus\{\0\}$ such that $\VR_{S}(\0)=\bigcap_{t\in T}h(\0,t)$.

Consider the hyperplane arrangement $\SetOf{x_i=0}{i\in [n]}$ in $\cH$.
Its maximal cells are in bijection with the $2^n-2$ ordered partitions of $[n]$ in two sets.
A partition $I\sqcup J=[n]$ corresponds to the \enquote{half-open} cone given by the hyperplanes $x_i > 0$ for $i\in I$ and $x_j\leq 0$ for $j\in J$.
We call $(I,J)$ the signature of the cone.

Denote by $S_{I,J}$ the points of $S\setminus\{\0\}$ that are contained in the cone with signature $(I,J)$.
For a point $a\in S_{I,J}$, the halfspace $h(\0,a)$ is given by the equation $\max_{i\in I}x_i\leq\max_{j\in J}(-a_j+x_j)$.

Let $b\in S_{I,J}$ such that $-\pi_J(b)\geq -\pi_J(a)$.
Then we have $\max_{j\in J}(-a_j+x_j)\leq\max_{j\in J}(-b_j+x_j)$ for every $x\in\cH$.
This implies the inclusion $h(\0,a)\subseteq h(\0,b)$.
In this case, $h(\0,b)$ does not contribute to the intersection $\bigcap_{s\in S\setminus\{\0\}}h(\0,s)$.
Therefore, the significant halfspaces come from the nondominated points of $-\pi_J(S_{I,J})$.
By Lemma \ref{lemma:supdisc-fingen}, there are only finitely many nondominated points, as $\pi_J(S_{I,J})$ is also super-discrete; we used Remark \ref{rmk:super-discrete}.

There could be infinitely many points that project onto a nondominated point of $-\pi_J(S_{I,J})$, but all of them give the same halfspace.
Indeed, from the above observations, $h(\0,a)=h(\0,b)$ is equivalent to $\pi_J(a)=\pi_J(b)$ for points $a,b\in S_{I,J}$.
Hence, we can select a finite subset $T_{I,J}$ of $S_{I,J}$ such that $\bigcap_{t\in T_{I,J}}h(\0,t)=\bigcap_{s\in S_{I,J}}h(\0,s)$.

All in all, we have $\VR_S(\0)=\bigcap_{\emptyset\neq I\subset [n]}\bigcap_{t\in T_{I,[n]\setminus I}}h(\0,t)$, which is an intersection of finitely many tropical halfspaces.
\end{proof}

We close this section with examples of asymmetric tropical Voronoi diagrams for a family of lattices.
These are interesting for several reasons; e.g., they provide counter-examples to several claims made in \cite{RRlattice}.

\begin{figure}[ht]
  \centering
  \begin{tikzpicture}[semithick,x={(0.866,-0.5)},y={(0,1)},z={(-0.866,-0.5)}]
  \small
    \definecolor{tolorange}{RGB}{238,119,51} 
    \definecolor{tolblue}{RGB}{0,119,187}    
    \definecolor{tolgreen}{RGB}{0,153,136}   
    \definecolor{tolpurple}{RGB}{136,34,85}  
    \definecolor{ibmyellow}{RGB}{255, 176, 0}  
    
    
    \fill[color=tolorange, opacity=0.25] (0,-2) -- (1,-1) -- (1,0) -- (0,-1) -- cycle;
    \fill[color=tolorange, opacity=0.25] (1,0) -- (2,1) -- (2,2) -- (1,1) -- cycle;
    
    \fill[color=tolorange, opacity=0.25] (0,-2) -- (1,-1) -- (1,0) -- (2,1) -- (2,2) -- (1,1) -- (-1,1) -- (-2,0) -- (-2,-1) -- (0,-1) -- cycle;
    
    \draw[color=tolorange, very thick] (0,-3) -- (0,-2) -- (1,-1) -- (1,0) -- (2,1) -- (2,2) -- (3, 3);
    \draw[color=tolorange, very thick] (3, 3) -- (1,1) -- (-1,1) -- (-2,0) -- (-2,-1) -- (0,-1) -- (0,-3);
    \draw[color=tolorange, very thick]  (0,-1) -- (1,0) -- (1,1);

    \node (V000) at (0, 0) [circle,fill,inner sep=1.5pt,color=tolpurple,draw = tolpurple!50!black, semithick, label={[tolpurple]above:$\bm 0$}] {};
    \node (V-12-1) at (3, 0) [circle,fill,inner sep=1.5pt,color=tolpurple,draw = ibmyellow, semithick, label={right:$b_1$}] {};
    \node at (-2, 2) [circle,fill,inner sep=1.5pt,color=tolpurple,draw = ibmyellow, semithick] {};
    \node at (1, 2) [circle,fill,inner sep=1.5pt,color=tolpurple,draw = ibmyellow, semithick] {};
    \node at (6, 0) [circle,fill,inner sep=1.5pt,color=tolpurple,draw = ibmyellow, semithick] {};
    \node at (-1, 4) [circle,fill,inner sep=1.5pt,color=tolpurple,draw = ibmyellow, semithick] {};
    \node at (2, 4) [circle,fill,inner sep=1.5pt,color=tolpurple,draw = ibmyellow, semithick] {};
    \node at (4, 2) [circle,fill,inner sep=1.5pt,color=tolpurple,draw = ibmyellow, semithick] {};
    \node at (-1, -2) [circle,fill,inner sep=1.5pt,color=tolpurple,draw = ibmyellow, semithick] {};
    \node at (-3, 0) [circle,fill,inner sep=1.5pt,color=tolpurple,draw = ibmyellow, semithick] {};
    \node (V2-20) at (-4, -2) [circle,fill,inner sep=1.5pt,color=tolpurple,draw = ibmyellow, semithick,label={above left:$b_0$}] {};
    \node at (-2, -4) [circle,fill,inner sep=1.5pt,color=tolpurple,draw = ibmyellow, semithick] {};
    \node at (1, -4) [circle,fill,inner sep=1.5pt,color=tolpurple,draw = ibmyellow, semithick] {};
    \node at (2, -2) [circle,fill,inner sep=1.5pt,color=tolpurple,draw = ibmyellow, semithick] {};
    \node at (-6, 0) [circle,fill,inner sep=1.5pt,color=tolpurple,draw = white, semithick] {};
    \node at (6, 0) [circle,fill,inner sep=1.5pt,color=tolpurple,draw = white, semithick] {};
    \node at (3, 6) [circle,fill,inner sep=1.5pt,color=tolpurple,draw = ibmyellow, semithick] {};
    \node at (-3, -6) [circle,fill,inner sep=1.5pt,color=tolpurple,draw = ibmyellow, semithick] {};
    \node at (-5, 2) [circle,fill,inner sep=1.5pt,color=tolpurple,draw = white, semithick] {};
    \node at (5, -2) [circle,fill,inner sep=1.5pt,color=tolpurple,draw = white, semithick] {};
    \node at (-5, -4) [circle,fill,inner sep=1.5pt,color=tolpurple,draw = ibmyellow, semithick] {};
    \node at (5, 4) [circle,fill,inner sep=1.5pt,color=tolpurple,draw = ibmyellow, semithick] {};
		\node at (6, 6) [circle,fill,inner sep=1.5pt,color=tolpurple,draw = white, semithick] {};
    
    \path[->] (V000) edge node[near start, above] {} (V2-20);
    \draw[->] (V000) edge node[very near start, above] {} (V-12-1);
    
    \node at (0, -3) [circle,fill,inner sep=1.5pt,color=tolblue,draw = white, semithick,label={[tolblue]below:$(1,1,-2)$}] {};
    \node at (1, -1) [circle,fill,inner sep=1.5pt,color=tolblue,draw = white, semithick,label={[tolblue]-30:$(0,1,-1)$}] {};
    \node at (2, 1) [circle,fill,inner sep=1.5pt,color=tolblue,draw = white, semithick,label={[tolblue]-30:$(-1,1,0)$}] {};
    \node at (3, 3) [circle,fill,inner sep=1.5pt,color=tolblue,draw = white, semithick,label={[tolblue]above:$(-2,1,1)$}] {};
    \node at (-1, 1) [circle,fill,inner sep=1.5pt,color=tolblue,draw = white, semithick,label={[tolblue]above:$(0,-1,1)$}] {};
    \node (V1-10) at (-2, -1) [circle,fill,inner sep=1.5pt,color=tolblue,draw = white, semithick,label={[tolblue]-120:$(1,-1,0)$}] {};

\end{tikzpicture}
  \caption{The Voronoi region $\VR(\0)$ in the lattice $L_2(2,1,1)$.
    The parts in a darker shade of orange are also contained in the Voronoi regions of $b_0+b_1=(1,0,-1)$ and $-b_0-b_1=(-1,0,1)$, respectively.
    The six blue points are the tropical vertices of $\VR(\0)$}
  \label{fig:RRlattice-counterexample}
\end{figure}
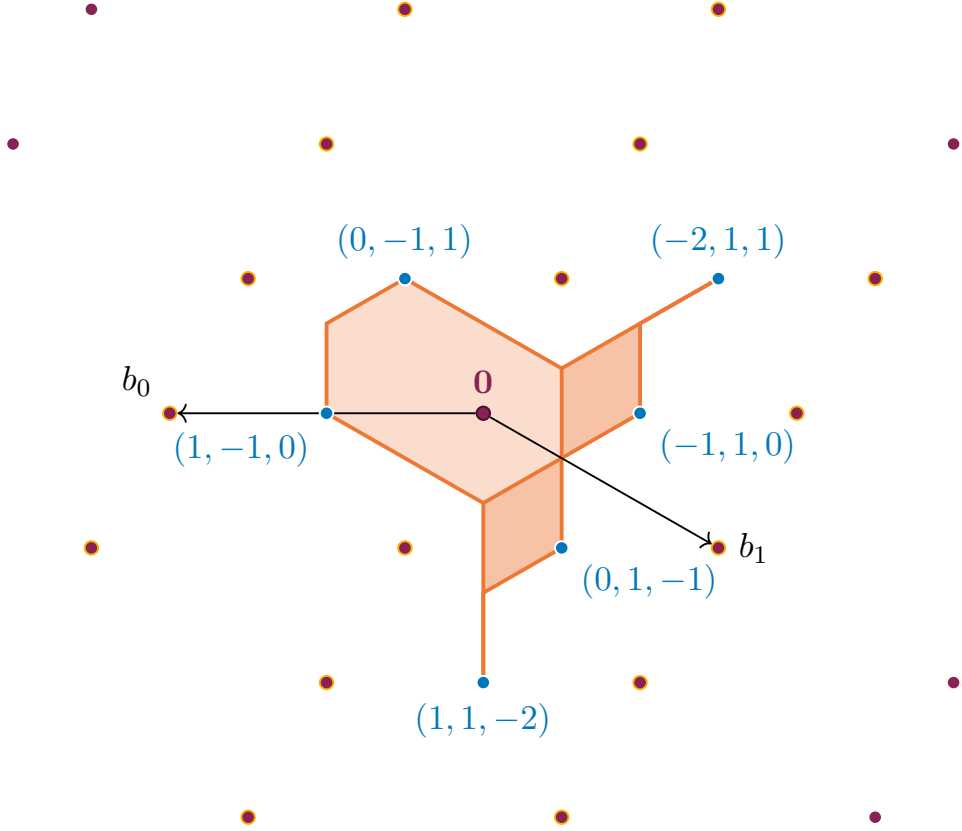

\begin{example} \label{exmp:lattice}
  In \cite[\textsection 6.4]{RRlattice}, the authors construct a lattice $L_2(\alpha,\gamma,\eta)$ in $\cH\subset\RR^3$ with basis $b_0=(\alpha,-\alpha,0)$ and $b_1=(-\gamma,\gamma+\eta,-\eta)$.
  Here the three parameters $\alpha,\gamma,\eta$ are positive integers with $\gamma<\alpha\leq\gamma+\eta$.
  We will show that the lattice $L_2(\alpha,\gamma,\eta)$ is graphical if $\alpha$ divides $\gamma+\eta$.
  To this end consider the $3{\times}3$-matrix
  \[
    Q \ = \ \begin{bmatrix}
      \alpha+\eta & -\alpha & -\eta \\
      -\alpha & \alpha & 0 \\
      -\eta & 0 & \eta
    \end{bmatrix} \enspace ,
  \]
  which is the Laplacian matrix of a multi-tree on three nodes; the first node is adjacent to the other two, with multiplicities $\alpha$ and $\eta$.
  We compute $(\alpha+\eta,-\alpha,-\eta)=\left(1+\frac{\gamma+\eta}{\alpha}\right)b_0+b_1$ and $(-\alpha,\alpha,0)=-b_0$.
  That invertible linear transformation of the basis is unimodular when $\alpha$ divides $\gamma+\eta$.
  So this furnishes a counter-example to \cite[Proposition~6.29]{RRlattice}, where it was claimed that $L_2(\alpha,\gamma,\eta)$ is not graphical.
  Observe that $\gamma$ does not occur in the matrix $Q$.

  Figure~\ref{fig:RRlattice-counterexample} displays the lattice $L_2(2,1,1)$ and the Voronoi region of the origin.
  This is a $\max$-tropical polytope with the six tropical vertices
  \begin{equation}\label{eq:RRlattice-counterexample:vertices}
    (1,-1,0) \,,\ \underline{(1,1,-2)}\,,\ \underline{(0,1,-1)}\,,\ (-1,1,0)\,,\ (-2,1,1)\,,\ (0,-1,1) \enspace ,
  \end{equation}
  in cyclic order.
  These are are local maxima of the distance function from the origin (called \enquote{critical points} in \cite{RRlattice}).
  Their reflection at $\0$ is not a translation of those points, whence that lattice is not strongly reflection invariant.
  This refutes \cite[Theorem~6.1]{RRlattice}.
  The same example also disproves \cite[Theorem~6.28]{RRlattice}:
  the lattice $L_2(2,1,1)$ is defined by a multi-tree on three vertices, but it is not strongly reflection invariant.

  Moreover, the same example also shows that \cite[Theorem~6.9 (ii)]{RRlattice} is false:
  the two underlined tropical vertices in \eqref{eq:RRlattice-counterexample:vertices} are missing in that statement.
  Notice that \cite[Theorem~6.9 (ii)]{RRlattice} is also used in the proof of \cite[Theorem~4]{LapLatticeGraph}.
\end{example}

\section{Power diagrams over fields of Puiseux series}
\label{sec:power}
\noindent
Our next goal is to relate asymmetric tropical Voronoi diagrams with the ordinary polyhedral geometry over ordered fields.
To this end, we consider the field of generalized dual Puiseux series $\hahnseries{\RR}{t}^*$; see \cite[\textsection 2.7]{ETC}.
Its elements are the formal power series in $t$, with real coefficients, such that the exponents form a strictly decreasing sequence of reals without finite accumulation points.
That field is ordered, and it is equipped with the dual valuation map
\[
  \val^* : \hahnseries{\RR}{t}^* \to \RR \,,\ \ \sum_{k=0}^\infty \alpha_k t^{r_k} \mapsto r_0 \enspace,
\]
where $\alpha_0\neq 0$ and $r_0>r_1>\cdots$, which sends a generalized dual Puiseux series to its highest exponent.
The sign of the generalized Puiseux series $\sum \alpha_k t^{r_k}$ is the sign of the leading coefficient $\alpha_0$.
The map $\val^*$ is surjective onto the reals, and it preserves the ordering, if restricted to non-negative generalized dual Puiseux series:
\[
  \val^*(\bm x) \leq \val^*(\bm y) \quad \text{if and only if} \quad \bm x \leq \bm y
\]
for $\bm x, \bm y\in\hahnseries{\RR}{t}^*_{\geq 0}$.
Notice that the ordinary Puiseux series with real coefficients have rational exponents, which are rising instead of falling, and the usual valuation map reverses the order, since it picks the lowest degree.
The compatibility with the order relations makes the generalized dual Puiseux series more convenient for our purposes.
In the sequel we abbreviate $\KK=\hahnseries{\RR}{t}^*$.

Pick a finite set of sites $\bm S\subset \KK^n$ and an arbitrary \emph{weight function} $w:\bm S\to\KK_{\geq 0}$.
Then our setup gives rise to the \emph{farthest power region} of $\bm a\in\bm S$, with respect to $w$, which is the set
\[
  \PR_{\bm S}^w(\bm a) \ := \ \bigSetOf{\bm x\in\KK_{\geq 0}^n}{\|\bm x-\bm a\|^2-w(\bm a) \geq \|\bm x-\bm b\|^2-w(\bm b) \text{ for all } \bm b\in \bm S} \enspace .
\]
The \emph{farthest power diagram} of $\bm S$ with respect to $w$ is the set of all intersections of the farthest power regions of $\bm S$, partially ordered by inclusion.
We denote it $\PD(S)$.
Farthest power diagrams are also called \enquote{maximal power diagrams} \cite{PowDiag}.
Of particular interest for our investigation is the map
\[
  \|\cdot\|:\KK^n \to \KK_{\geq 0} \,,\ \bm x \mapsto \sqrt{\sum_{i\in[n]} \bm x_i^2} \enspace ,
\]
which is called the \enquote{Euclidean norm} in \cite[p. 83]{BasuPollackRoy:2006}.
It is well defined because the generalized dual Puiseux series form a real closed field \cite{Markwig:2010}.
Yet, this is not a norm on the infinite-dimensional real vector space $\KK^n=(\hahnseries{\RR}{t}^*)^n$ in the usual sense, as its values lie in a proper superfield of the reals.

We abbreviate $\bm h^w(\bm a,\bm b)=\PR_{\{\bm a,\bm b\}}^w(\bm a)$.
Occasionally, we will also call the intersection $\bisector(\bm a,\bm b)=\bm h^w(\bm a,\bm b)\cap\bm h^w(\bm b, \bm a)$ a \emph{(two point) bisector}.
\begin{proposition}\label{prop:power-region}
  For the special weight function $w(\bm a)=\|\bm a\|^2$ the farthest power region $\bm h^w(\bm a,\bm b)$ for two sites is a linear halfspace in $\KK^n$.
  Consequently, for an arbitrary finite set of sites, $\bm S$, the power region $\PR_{\bm S}^w(\bm a)$ is a polyhedral cone.
\end{proposition}
\begin{proof}
  We have $\|\bm x-\bm a\|^2-\|\bm a\|^2 = \|\bm x\|^2-2\sum_{i\in[n]}\bm x_i \bm a_i$, whence the inequality $\|\bm x-\bm a\|^2-w(\bm a) \geq \|\bm x-\bm b\|^2-w(\bm b)$ is equivalent to $\sum_{i\in[n]}\bm a_i \bm x_i \leq \sum_{i\in[n]}\bm b_i \bm x_i$.
  For general $\bm S$ the power region is then described by finitely many homogeneous linear inequalities.
\end{proof}
Observe that the proof above only exploits the fact that $\|\cdot\|^2$ is a quadratic form.
The real-closedness of the field $\KK$ is irrelevant here.
In our power diagram notation we usually omit the upper index \enquote{$w$} if $w=\|\cdot\|^2$.

To see $d_\triangle$ as a distance function requires to pass from $\RR^n$ to the quotient $\torus{n}$.
Yet here we want to work in a tropically inhomogeneous setting.
This will allow us to state our first main result in a particularly concise form.
In $\KK^n$ we consider
\[
  \bm \cH \ := \bigSetOf{\bm x\in\KK_{>0}^n}{ \bm x_1 \bm x_2\cdots \bm x_n = 1 } \enspace ,
\]
which is the intersection of an affine algebraic hypersurface over $\KK$ with the positive orthant.
In differential geometry the hypersurface $\bm\cH$ occurs as a hyperbolic affine hypersphere \cite[Example 3.1]{AffDiffGeomHypsurf}.
This tropicalizes to the tropical hypersurface $\val^*(\bm \cH)=\cH$.
Further, the ray $\KK_{\geq 0} \cdot \bm x$, for $\bm x\in\KK_{> 0}^n$, intersects $\bm\cH$ in a unique point.
We obtain a commutative diagram, where the horizontal maps are embeddings and canonical projections, respectively:
\[
  \begin{tikzcd}
    \bm\cH \ar[r, hook] \ar[d, "\val^*"] & \KK_{>0}^n \ar[r, twoheadrightarrow] \ar[d, "\val^*"] & \KK_{>0}^n/\KK_{>0} \ar[d, "\val^*"]  \\
    \cH \ar[r, hook] & \RR^n \ar[r, twoheadrightarrow] & \torus{n}
  \end{tikzcd}
\]

\begin{lemma} \label{lemma:polyPR}
  Let $S\subset\cH$ be a super-discrete set of sites in general position, and let $\bm S\subset\bm\cH$ be a lifted point configuration such that $\val^*:\bm S\to -S$ is bijective.
  Then the cells of $\PD(\bm S)$ are polyhedral.
\end{lemma}

\begin{proof}
Consider $I\sqcup J=[n]$ an ordered partition of $[n]$ and $C_{I,J}$ the open polyhedron given by the equations $x_i>s_i$ for $i\in I$ and $x_j<s_j$ for $j\in J$. The set $-\pi_J(C_{I,J}\cap S)$ has finitely many nondominated points due to $S$ being super-discrete and Lemma \ref{lemma:supdisc-fingen}.
Consider $S_{I,J}$ the set of points in $C_{I,J}\cap S$ that project to the nondominated points of $-\pi_J(C_{I,J}\cap S)$.
By general position no two points project onto the same nondominated point, so $S_{I,J}$ is finite.

Let $T$ be the union of all the sets $S_{I,J}$ over all ordered partitions $I\sqcup J=[n]$.
The set $T$ is finite and contains all the Voronoi neighbors of $s$, due to general position.
The last condition implies
\[\VR_S(s) \ = \ \bigcap_{t\in T}h(s,t) \enspace.\]

Consider an arbitrary site $u$ from $S\setminus(T\cup\{s\})$ and $\bm u\in\bm S$ such that $\val^*(\bm u)=-u$.
We show that
\[ \bigcap_{\bm t\in\bm S\cap(\val^*)^{-1}(-T)}\bm h(\bm s, \bm t) \ \subseteq \ \bm h(\bm s, \bm u) \enspace, \]
which will imply that $\PR(\bm s)$ is a polyhedral cone with hyperplane description given by
\[ \PR(\bm s) \ = \ \bigcap_{\bm t\in\bm S\cap(\val^*)^{-1}(-T)}\bm h(\bm s, \bm t) \enspace . \]

Let $K:=\smallSetOf{i\in [n]}{u_i>s_i}$ and $L:=\smallSetOf{j\in [n]}{u_j<s_j}$.
By the construction of $T$, there exists $t\in T\cap C_{K,L}$ such that $t_j>u_j$ for all $j\in L$.
Pick $\bm t\in\bm S$ such that $\val^*(\bm t)=-t$.

Let $\bm x\in\bm h(\bm s, \bm t)$ arbitrary.
This is a point satisfying $\sum_{i\in K}(\bm s_i-\bm t_i)\cdot\bm x_i\leq\sum_{j\in L}(\bm t_j-\bm s_j)\cdot\bm x_j$.
In the following expression we use the notation $x=\val^*(\bm x)$.
Since $u\in C_{K,L}$ we obtain the equalities
\begin{equation}\label{eq:val1}
\val^*\left(\sum_{i\in K}(\bm s_i-\bm u_i)\cdot\bm x_i\right) \ = \ \max_{i\in K}(-s_i+x_i)
\end{equation}
and
\begin{equation}\label{eq:val2}
\val^*\left(\sum_{j\in L}(\bm u_j-\bm s_j)\cdot\bm x_j\right) \ = \ \max_{j\in L}(-u_j+x_j) \enspace.
\end{equation}
But $\bm x\in\bm h(\bm s, \bm t)$ implies $\max_{i\in K}(-s_i+x_i)\leq\max_{j\in L}(-t_j+x_j)$.
The latter is strictly smaller than $\max_{j\in L}(-u_j+x_j)$ in view of our selection for $t$.
This entails $\max_{i\in K}(-s_i+x_i)<\max_{j\in L}(-u_j+x_j)$. Using the last inequality with (\ref{eq:val1}) and (\ref{eq:val2}), we obtain
\[\sum_{i\in K}(\bm s_i-\bm u_i)\cdot\bm x_i \ < \ \sum_{j\in L}(\bm u_j-\bm s_j)\cdot\bm x_j \enspace .\]

The choice of $\bm x$ was arbitrary, so $\bm h(\bm s, \bm t)\subseteq\bm h(\bm s, \bm u)$.
Consequently,
\[\bigcap_{\bm t'\in\bm S\cap(\val^*)^{-1}(-T)}\bm h(\bm s, \bm t') \ \subseteq \ \bm h(\bm s, \bm u) \enspace .\]
\end{proof}

In the following main result the assumption on general position allows arbitrary lifts for the sites.
This is similar to the relationship between tropical and Puiseux polyhedra; e.g., see \cite[Corollary 8.15]{ETC}.
\begin{theorem} \label{th:PuiseuxLift}
  Let $S\subset\cH$ be a nonempty super-discrete set of sites in general position.
  Further, let $\bm S\subset\bm\cH$ be a lifted point configuration such that $\val^*:\bm S\to -S$ is bijective.
  Then $\val^*$ maps each farthest power region onto the corresponding Voronoi region, and this induces a poset isomorphism from the farthest power diagram $\PD(\bm S)$ to the asymmetric tropical Voronoi diagram $\VD(S)$.
\end{theorem}
\begin{proof}
  Pick a lifted point configuration $\bm S$ on the hypersurface $\bm\cH$, i.e., $\val^*(\bm S)=-S$; recall that $-\cH=\cH$.
  For $\bm a, \bm b \in \bm S$, by Proposition~\ref{prop:power-region}, the power region $\bm h(\bm a,\bm b)$ is determined by the linear inequality
  \begin{equation}\label{eq:PR-pair:ineq}
    \sum_{i\in[n]}\bm a_i \bm x_i \ \leq \ \sum_{i\in[n]}\bm b_i \bm x_i \enspace . 
  \end{equation}
  As $\bm a,\bm b \in \bm \cH$, the tropicalization of \eqref{eq:PR-pair:ineq} reads
  \begin{equation}
    \bigoplus_{i\in[n]}(-a_i + x_i) \ \leq \ \bigoplus_{i\in[n]} (-b_i + x_i) \enspace ,
  \end{equation}
  which is the defining inequality for $h(a,b)$.
  This yields $\val^*(\bm h(\bm a,\bm b))=h(a,b)$.
	
  The proof of Lemma~\ref{lemma:polyPR} implies that there is a finite subset $\bm S_a\subseteq\bm S\setminus\{\bm a\}$ such that $\PR(\bm a)=\bigcap_{\bm b\in\bm S_a}\bm h^w(\bm a,\bm b)$ and $\VR(a)=\bigcap_{b\in S_a} h(a,b)$, where $-S_a=\val^*(\bm S_a)$.
	
  Moreover, we have:
  \[
    \val^*(\PR(\bm a)) \ \subseteq \ \bigcap_{\bm b\in\bm S_a}\val^*(\bm h(\bm a,\bm b)) \ = \ \bigcap_{b\in S_a}h(a,b) \ = \ \VR(a) \enspace .
  \]

  To obtain equality, it suffices to prove that the pair of matrices defining $\VR(a)$ is tropically sign generic \cite[Theorem~8.12]{ETC}.
  For two distinct sites $a$ and $b$ we consider $I(a,b):=\smallSetOf{k\in[n]}{a_k<b_k}$.
  General position then implies $I(b,a)=[n]\setminus I(a,b)$.
	
  The pair of matrices $(A^-,A^+)\in\TT^{S_a\times[n]}$ which describes $\VR(a)$ as a tropical polyhedron has the entries
  \begin{align*}
    A^{-}_{b,k} \ &= \ 
                    \begin{cases}
                      -a_k & \text{if }k\in I(a,b) \\
                      -\infty & \text{if }k\in I(b,a)
                    \end{cases} \\
    \intertext{and}
    A^{+}_{b,k} \ &= \
                    \begin{cases}
                      -\infty & \text{if }k\in I(a,b) \\
                      -b_k & \text{if }k\in I(b,a) \enspace .
                    \end{cases} 
  \end{align*}
  The matrix $A:=A^-\oplus A^+$ has finite entries due to general position.
  Suppose that the pair $(A^-,A^+)$ is tropically sign singular.
  Then there exists a nonempty set $B\subseteq S_a$ and another set $K\subseteq [n]$ such that $|B|=|K|$ and $\tdet A^-_{B,K}=\tdet A^+_{B,K}=\tdet A_{B,K}$.
  This is equivalent to the existence of bijections $\mu,\nu:B\rightarrow K$ with $\mu(b)\in I(a,b)$ and $\nu(b)\in I(b,a)$, for $b\in J$, as well as
  \begin{equation}\label{eq:st}
    -\sum_{b\in B}a_{\mu(b)} \ = \ \tdet A \ = \ -\sum_{b\in B}b_{\nu(b)} \enspace .
  \end{equation}
  However, we have $\sum_{b\in B}b_{\nu(b)}<\sum_{b\in B}a_{\nu(b)}=\sum_{k\in K}a_{k}=\sum_{b\in B}a_{\mu(b)}$, where the inequality comes from $\nu(j)\in I(b,a)$, and the two equalities use that $\mu$ and $\nu$ are bijections.
  We arrive at a contradiction, which refutes \eqref{eq:st}.
  Hence, $(A^-,A^+)$ is tropically sign generic.
\end{proof}

\begin{remark}
  Observe that the tropicalization $\val^*(\bm S)=-S$ to the negative is natural here as we are mapping points with Puiseux coordinates to (apices of) tropical halfspaces.
\end{remark}

\begin{remark}
  Instead of farthest power diagrams and lower convex hulls (for defining regular subdivisions), we could use ordinary power diagrams and upper convex hulls.
  This goes hand in hand with exchanging the arguments in the asymmetric tropical distance function.
\end{remark}

\section{A different view}
\label{sec:general}
\noindent
Edelsbrunner and Seidel~\cite{Edelsbrunner+Seidel:1986} studied Voronoi diagrams for general metrics.
In general, their construction is geometrically different from the approach of Amini and Manjunath \cite{RRlattice}, which we adopt here.
Yet for a discrete set $S$ in general position the two concepts agree, and this is what we will show now.

Following \cite[\textsection 3]{Edelsbrunner+Seidel:1986} we define a function $D_S:\cH\to S$ by letting
\begin{equation}\label{eq:voronoi-cell}
  D_S(x) \ := \ \bigSetOf{ a\in S }{ d_\triangle(x,a) = \min_{b\in S}d_\triangle(x,b) } \enspace.
\end{equation}
That function defines an equivalence relation $\equiv_S$ on $\cH$ via: $x\equiv_S y$ if and only if $D_S(x)=D_S(y)$.
The equivalence classes of $\equiv_S$ partition $\cH$, and they are called \emph{V-cells} with respect to $S$.
For $T\subseteq S$ we set
\[
  V_T \ := \ \bigSetOf{ x\in\cH }{ D_S(x)=T } \enspace ,
\]
and this is a V-cell or empty.
A V-cell $V_T$ such that $T=\{a\}$ is a singleton is the \emph{V-region} of the site $a$.
Clearly, if $S$ is finite, then there are only finitely many V-cells.

\begin{remark}\label{rmk:cardTgenpos}
	If a V-cell $V_T$ is nonempty and $S$ is in general position, then $|T|\leq n$.
	Indeed, if $|T|>n$ and there existed a point $x\in V_T$, then the pigeonhole principle would imply the existence of an index $i\in[n]$ and of at least two distinct sites $s,t\in T$ such that $d_\triangle(x,s)=n(x_i-s_i)$ and $d_\triangle(x,t)=n(x_i-t_i)$.
	Since $x\in V_T$, we must have $n(x_i-s_i)=d_\triangle(x,s)=d_\triangle(x,t)=n(x_i-t_i)$ which entails $s_i=t_i$.
	This cannot happen, as we assumed $S$ to be in general position.
	
	Similarly, any Voronoi cell $\bigcap_{a\in T}\VR_S(a)$ is empty when $|T|>n$ and $S$ is in general position.
\end{remark}

\begin{lemma}\label{lem:region-generic}
  Let $S\subset\cH$ be a discrete set of sites.
  Then the topological closure of the V-region $V_{\{a\}}$ is contained in the Voronoi region $\VR_S(a)$.
  If $S$ is in general position, then the closure of $V_{\{a\}}$ equals $\VR_S(a)$.
\end{lemma}
\begin{proof}
  Let $x\in V_{\{a\}}$, i.e., $D_S(x)=\{a\}$.
  Then $d_\triangle(x,a) \leq d_\triangle(x,b)$ for all $b\in S$, whence $x\in\VR_S(a)$.
  That is, $V_{\{a\}}\subseteq\VR_S(a)$.
  Since $\VR_S(a)$ is closed it also contains the closure of $V_{\{a\}}$.

  For the converse we assume that $S$ is in general position, and we pick a point $x$ in the interior of the Voronoi region $\VR_S(a)$.
  Then it is a consequence of Lemma~\ref{lem:bisector-generic} that $d_\triangle(x,a)<d_\triangle(x,b)$ for all sites $b\neq a$.
  That is, $D_S(x)=\{a\}$ or, equivalently, $x\in V_{\{a\}}$.
  The conclusion now follows from Lemma~\ref{lem:puredim}.
\end{proof}

For the sake of conciseness we call the topological closure of a V-cell a \emph{closed V-cell}.
Now we can prove that our tropical Voronoi diagrams agree with the construction in \cite{Edelsbrunner+Seidel:1986}, provided that the sites are in general position.
This is based on the crucial fact that tropical hyperplane arrangements in general position essentially behave like ordinary hyperplanes in general position; cf.\ \cite[\textsection\textsection 7.5, 8.3]{ETC}.
Observe that tropical polyhedra are ordinary polyhedral complexes, which thus have a dimension.
This dimension agrees with the notion of \enquote{tropical rank} \cite[\textsection 5.3]{Tropical+Book}.

\begin{theorem}\label{thm:V-cells}
  Let $S\subset\cH$ be a discrete set of sites in general position.
  Then the Voronoi cells in $\VD(S)$ are precisely the closed V-cells with respect to $S$.
  Moreover, the V-cell $V_T$ (and its closure) is of dimension $n-|T|$.
\end{theorem}
\begin{proof}
	
  Let $V_T$ be a V-cell.
  From its definition, it is clear that the closure of $V_T$ is a subset of $W_T:=\bigcap_{a\in T}\VR_S(a)$.
  In particular, if $W_T$ is empty, then also $V_T$ is empty.
	
  Now consider a nonempty set $T\subseteq S$ such that $W_T$ is a Voronoi cell, i.e., it is not empty.
  We need to show that $W_T$ is the closure of the V-cell $V_T$.
  The case when $T$ is a singleton is covered by Lemma~\ref{lem:region-generic}, so we will assume that $|T|\geq 2$ from now on.
  Note that $|T|\leq n$ because $S$ is in general position; see Remark~\ref{rmk:cardTgenpos}.
  Since $T$ contains at least two sites, we can pick any $a\in T$ and write $W_T$ as the intersection $\bigcap_{b\in T\setminus\{a\}}\bisector(a,b)$.
  Now the general position comes in twice.
  First, by Lemma~\ref{lem:bisector-generic} the bisectors are boundary planes of tropical halfspaces.
  Second, the tropical halfspace arrangement induced by those boundary planes is generic.
  Hence, via lifting to Puiseux series that tropical halfspace arrangement has the same intersection poset as an ordinary hyperplane arrangement over Puiseux series; see \cite[\textsection 8.3]{ETC}.
  That is, the set $W_T$ is a complete intersection, and thus $W_T$ is the closure of $V_T$.
  This also proves that $W_T$ is an ordinary polyhedral complex of pure dimension $n-|T|$.
\end{proof}

The proof above gives a somewhat high-level view on the situation.
The key idea is to employ lifts to Puiseux series, much like in Section~\ref{sec:power} above.

Boissonnat et al.\ \cite[\textsection 8.2]{BSTY1998} developed an incremental algorithm for computing Voronoi diagrams in the sense of \cite{Edelsbrunner+Seidel:1986} for simplicial distance functions.
By \cite[Theorem~5.1]{BSTY1998} there are at most $\Theta(m^{\lceil (n-1)/2\rceil})$ many V-cells, for $n$ fixed.
In  \cite[Theorem~8.8]{BSTY1998} the authors show that the tropical Voronoi diagram of $m$ sites in $\torus{n}$ in general position can be constructed incrementally in randomized expected time $O(m \log m + m^{\lceil (n-1)/2\rceil})$.
This agrees with the complexity to compute Euclidean Voronoi diagrams \cite[Corollary~17.2.6]{BoissonnatYvinec:1998}.
However, it is faster than the algorithm in \cite[Theorem~10]{TropBis} for computing tropical Voronoi diagrams with respect to the symmetric tropical distance; that expected time complexity bound is $O(m^{n-1}\log m)$.
Note that the algorithms of \cite{BoissonnatYvinec:1998} and \cite{TropBis} produce different types of output.

\section{Delone complexes and Laurent monomial modules}
\label{sec:delaunay}
\noindent
For super-discrete sites in general position, by Theorem \ref{th:PuiseuxLift}, the combinatorial type of the asymmetric tropical Voronoi diagrams is preserved in the lift to generalized dual Puiseux series $\KK=\hahnseries{\RR}{t}^*$.
Recall that this combinatorial type is defined as the intersection poset of the Voronoi regions.
Here we show how such posets occur in commutative algebra.
To this end we consider the \emph{monomial lifting function}
\begin{equation}
  \mu : \cH\to\bm\cH \,,\ s\mapsto t^{-s} \enspace,
\end{equation}
where $t^{-s}=(t^{-s_1},\dots,t^{-s_n})$.
First we assume that $S$ is a finite set of sites and $\bm S=\mu(S)\subset\bm\cH$ is its \emph{monomial lift}.
Then $\PD(\bm S)$ gives rise to a family of power diagrams $\PD(\bm S(t))$ over the reals which depend on a real parameter $t$.
The cells of a power diagram, over $\KK$ or $\RR$, are partially ordered by inclusion.

\begin{lemma}\label{lem:PDS-vs-PDSt}
  Let $S\subset\cH$ be a finite set of sites in general position with monomial lift $\bm S=\mu(S)$.
  Then the Puiseux farthest power diagram $\PD(\bm S)$ is isomorphic to the boundary complex of an ordinary polyhedron over the ordered field\/ $\KK$ of real Puiseux series.
  Moreover, for any $t$ large enough, $\PD(\bm S)$ and $\PD(\bm S(t))$ are isomorphic as posets.
\end{lemma}
\begin{proof}
  The first claim is a consequence of \cite[\textsection 4]{PowDiag}, which discusses power diagrams over the reals.
  For a discussion of polyhedral geometry over arbitrary ordered fields, see \cite[Appendix~A]{ETC}.
  The second claim follows from \cite[\textsection 8.5]{ETC}.
  It is explained in \cite[Theorem~12]{ABGJ:2018} how to find a real number $t_0>1$ such that $\PD(\bm S)$ and $\PD(\bm S(t))$ are isomorphic as posets for all $t>t_0$.
\end{proof}
Lemma~\ref{lem:PDS-vs-PDSt} holds for more general lifting functions.
Yet, by restricting to the specific function $\mu$ we avoid questions concerning convergence of Puiseux series.
Moreover, properties of the function $\mu$ enter the quantitative analysis in \cite[Theorem~12]{ABGJ:2018}.

Now we pass to the case, where $S$ is both super-discrete and an $(r,R)$-system, with $0<r<R\leq\infty$, but not necessarily in general position.
Recall that this includes the situation in which $S$ is an arbitrary finite set of sites.
Then we can dualize the asymmetric tropical Voronoi diagrams as follows.
\begin{definition}\label{def:delaunay}
  The \emph{(asymmetric tropical) Delone complex} $\Del(S)$ is defined as the clique complex of the dual graph of the Voronoi diagram $\VD(S)$.
\end{definition}
The nodes of the dual graph of $\VD(S)$ are the Voronoi regions, and they are adjacent if their intersection has codimension at most one.
A \emph{clique} in a graph is a subset of the nodes such that any two nodes in the set are connected by an edge.
The cliques form an abstract simplicial complex, which is called the \emph{clique complex}.
Simplicial complexes which arise in this way are also called \enquote{flag simplicial complexes}.
Delone complexes do not need to be pure. 

Again we let $\bm S=\mu(S)$ be the monomial lift.
Now we additionally assume that $S$ is in general position, while we also keep our previous assumption that $S$ is a super-discrete $(r,R)$-system.
By Theorem \ref{th:polySuperDisc} each Voronoi region $\VR(s)$ is a $\max$-tropical polyhedron, and by Theorem~\ref{th:PuiseuxLift} the monomial lift of $\VR(s)$ is an ordinary polyhedron over $\KK$.
We conclude that $\PD(\bm S)$ is a polyhedral complex over $\KK$, which may be infinite.
Similarly, $\PD(\bm S(t))$ is a polyhedral complex over the reals.
Thanks to $S$ being an $(r,R)$-system there is a uniform bound for $t$ relative to every cell.
Note that the upper bound $R$ is not required in the argument, whence that uniform bound exists for $(r,\infty)$-systems, too.
It follows that Lemma~\ref{lem:PDS-vs-PDSt} is valid for super-discrete $(r,R)$-systems in general position.

\begin{lemma}\label{lem:delaunay}
  Let $S\subset\cH$ be a super-discrete $(r,R)$-system in general position, for $0<r<R\leq\infty$.
  Then the Delone complex $\Del(S)$ is dual to $\VD(S)$ as a partially ordered set.
\end{lemma}
\begin{proof}
  The results from Section~\ref{sec:general} apply, and we use the map $D_S:\cH\to S$ defined in \eqref{eq:voronoi-cell}.
  In view of Theorem~\ref{thm:V-cells} the set $T\subseteq S$ forms a cell of $\Del(S)$ if and only if the closure of $V_T$ is a Voronoi cell.
\end{proof}

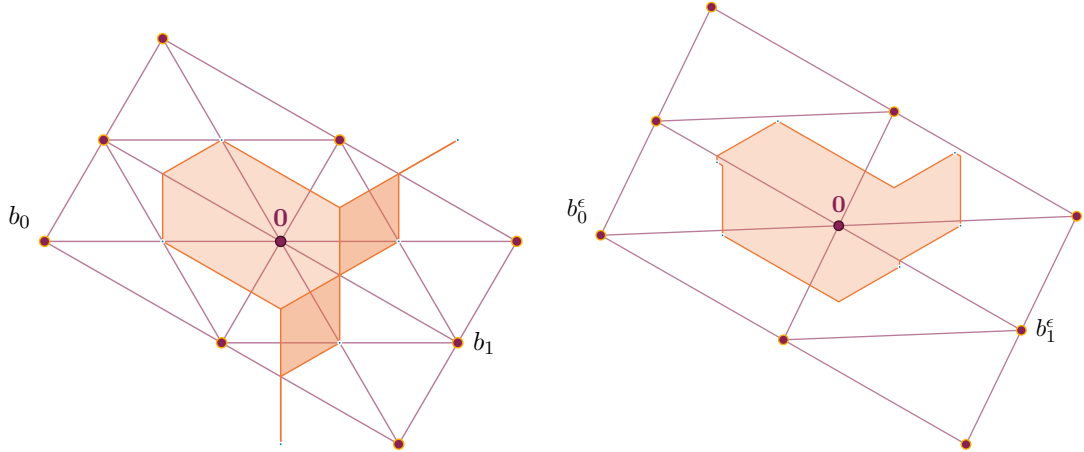
\begin{figure}[th]
  \begin{tabular}{cc}
    \begin{tikzpicture}[semithick,x={(0.866,-0.5)},y={(0,1)},z={(-0.866,-0.5)}]
  \small
    \definecolor{tolorange}{RGB}{238,119,51} 
    \definecolor{tolblue}{RGB}{0,119,187}    
    \definecolor{tolgreen}{RGB}{0,153,136}   
    \definecolor{tolpurple}{RGB}{136,34,85}  
    \definecolor{ibmyellow}{RGB}{255, 176, 0}  
    
    
    \draw[tolpurple,opacity = 0.6] (3,0) -- (4,2) -- (1,2) -- (-2,2) -- (-3,0) -- (-4,-2) -- (-1,-2) -- (2,-2) -- cycle;
    \draw[tolpurple,opacity = 0.6] (3,0) -- (0,0) -- (4,2);
    \draw[tolpurple,opacity = 0.6] (1,2) -- (0,0) -- (-3,0);
    \draw[tolpurple,opacity = 0.6] (-4,-2) -- (0,0) -- (-1,-2);
    \draw[tolpurple,opacity = 0.6] (-2,2) -- (0,0) -- (2,-2);
    \draw[tolpurple,opacity = 0.6] (-1,-2) -- (3,0) -- (1,2) -- (-3,0) -- cycle;
    
    \fill[color=tolorange, opacity=0.25] (0,-2) -- (1,-1) -- (1,0) -- (0,-1) -- cycle;
    \fill[color=tolorange, opacity=0.25] (1,0) -- (2,1) -- (2,2) -- (1,1) -- cycle;

    \fill[color=tolorange, opacity=0.25] (0,-2) -- (1,-1) -- (1,0) -- (2,1) -- (2,2) -- (1,1) -- (-1,1) -- (-2,0) -- (-2,-1) -- (0,-1) -- cycle;
    
    \draw[color=tolorange] (0,-3) -- (0,-2) -- (1,-1) -- (1,0) -- (2,1) -- (2,2) -- (3, 3);
    \draw[color=tolorange] (3, 3) -- (1,1) -- (-1,1) -- (-2,0) -- (-2,-1) -- (0,-1) -- (0,-3);
    \draw[color=tolorange]  (0,-1) -- (1,0) -- (1,1);

    \node (V000) at (0, 0) [circle,fill,inner sep=1.5pt,color=tolpurple,draw = tolpurple!50!black, semithick, label={[tolpurple]above:$\mathbf 0$}] {};
    \node (V-12-1) at (3, 0) [circle,fill,inner sep=1.5pt,color=tolpurple,draw = ibmyellow, semithick, label={right:$b_1$}] {};
    \node at (-2, 2) [circle,fill,inner sep=1.5pt,color=tolpurple,draw = ibmyellow, semithick] {};
    \node at (1, 2) [circle,fill,inner sep=1.5pt,color=tolpurple,draw = ibmyellow, semithick] {};
    \node at (4, 2) [circle,fill,inner sep=1.5pt,color=tolpurple,draw = ibmyellow, semithick] {};
    \node at (-1, -2) [circle,fill,inner sep=1.5pt,color=tolpurple,draw = ibmyellow, semithick] {};
    \node at (-3, 0) [circle,fill,inner sep=1.5pt,color=tolpurple,draw = ibmyellow, semithick] {};
    \node (V2-20) at (-4, -2) [circle,fill,inner sep=1.5pt,color=tolpurple,draw = ibmyellow, semithick,label={above left:$b_0$}] {};
    \node at (2, -2) [circle,fill,inner sep=1.5pt,color=tolpurple,draw = ibmyellow, semithick] {};
    
    
    \node at (0, -3) [circle,fill,inner sep=0.5pt,color=tolblue,draw = white, semithick] {};
    \node at (1, -1) [circle,fill,inner sep=0.5pt,color=tolblue,draw = white, semithick] {};
    \node at (2, 1) [circle,fill,inner sep=0.5pt,color=tolblue,draw = white, semithick] {};
    \node at (3, 3) [circle,fill,inner sep=0.5pt,color=tolblue,draw = white, semithick] {};
    \node at (-1, 1) [circle,fill,inner sep=0.5pt,color=tolblue,draw = white, semithick] {};
    \node (V1-10) at (-2, -1) [circle,fill,inner sep=0.5pt,color=tolblue,draw = white, semithick] {};

\end{tikzpicture} &
    \begin{tikzpicture}[semithick,x={(0.866,-0.5)},y={(0,1)},z={(-0.866,-0.5)}]
  \small
    \definecolor{tolorange}{RGB}{238,119,51} 
    \definecolor{tolblue}{RGB}{0,119,187}    
    \definecolor{tolgreen}{RGB}{0,153,136}   
    \definecolor{tolpurple}{RGB}{136,34,85}  
    \definecolor{ibmyellow}{RGB}{255, 176, 0}  
    
    
    \draw[tolpurple,opacity = 0.6] (3.3,0) -- (4.3,2.3) -- (1,2.3) -- (-3.3,0) -- (-4.3,-2.3) -- (-1,-2.3) -- cycle;
    \draw[tolpurple,opacity = 0.6] (3.3,0) -- (0,0) -- (4.3,2.3);
    \draw[tolpurple,opacity = 0.6] (1,2.3) -- (0,0) -- (-3.3,0);
    \draw[tolpurple,opacity = 0.6] (-4.3,-2.3) -- (0,0) -- (-1,-2.3);
		\draw[tolpurple,opacity = 0.6] (1,2.3) -- (-2.3,2.3) -- (-3.3,0);
    \draw[tolpurple,opacity = 0.6] (-1,-2.3) -- (2.3,-2.3) -- (3.3,0);
		
    \fill[color=tolorange, opacity=0.25] (0,-1.2) -- (1.1,-0.1) -- (1.1,0) -- (2.2,1.1) -- (2.2,2.2) -- (2.1,2.2) -- (1,1.1) -- (-1.1,1.1) -- (-2.2,0) -- (-2.2,-0.1) -- (-2.1,-0.1) -- (-2.1,-1.2) -- cycle;
    
    \draw[tolorange] (0,-1.2) -- (1.1,-0.1) -- (1.1,0) -- (2.2,1.1) -- (2.2,2.2) -- (2.1,2.2) -- (1,1.1) -- (-1.1,1.1) -- (-2.2,0) -- (-2.2,-0.1) -- (-2.1,-0.1) -- (-2.1,-1.2) -- cycle;
    
    \node (V000) at (0, 0) [circle,fill,inner sep=1.5pt,color=tolpurple,draw = tolpurple!50!black, semithick, label={[tolpurple]above:$\mathbf 0$}] {};
    \node (V-12-1) at (3.3, 0) [circle,fill,inner sep=1.5pt,color=tolpurple,draw = ibmyellow, semithick, label={right:$b_1^\epsilon$}] {};
    \node at (-2.3, 2.3) [circle,fill,inner sep=1.5pt,color=tolpurple,draw = ibmyellow, semithick] {};
    \node at (1, 2.3) [circle,fill,inner sep=1.5pt,color=tolpurple,draw = ibmyellow, semithick] {};
    \node at (4.3, 2.3) [circle,fill,inner sep=1.5pt,color=tolpurple,draw = ibmyellow, semithick] {};
    \node at (-1, -2.3) [circle,fill,inner sep=1.5pt,color=tolpurple,draw = ibmyellow, semithick] {};
    \node at (-3.3, 0) [circle,fill,inner sep=1.5pt,color=tolpurple,draw = ibmyellow, semithick] {};
    \node (V2-20) at (-4.3, -2.3) [circle,fill,inner sep=1.5pt,color=tolpurple,draw = ibmyellow, semithick,label={above left:$b_0^\epsilon$}] {};
    \node at (2.3, -2.3) [circle,fill,inner sep=1.5pt,color=tolpurple,draw = ibmyellow, semithick] {};
    
    
    \node at (-2.2, -0.1) [circle,fill,inner sep=0.5pt,color=tolblue,draw = white, semithick] {};
    \node at (1.1, -0.1) [circle,fill,inner sep=0.5pt,color=tolblue,draw = white, semithick] {};
    \node at (2.2, 1.1) [circle,fill,inner sep=0.5pt,color=tolblue,draw = white, semithick] {};
    \node at (2.1, 2.2) [circle,fill,inner sep=0.5pt,color=tolblue,draw = white, semithick] {};
    \node at (-1.1, 1.1) [circle,fill,inner sep=0.5pt,color=tolblue,draw = white, semithick] {};
    \node at (-2.1, -1.2) [circle,fill,inner sep=0.5pt,color=tolblue,draw = white, semithick] {};

    

    
\end{tikzpicture} 
  \end{tabular}
  \caption{Left: Delone complex of nine sites arising from Example~\ref{exmp:lattice}.  Right: Delone complex of the generic perturbation by $\epsilon=1/10$}
  \label{fig:Del}
\end{figure}

We want to describe the Puiseux power diagram $\PD(\bm S)$.
In our case, according to \cite[\textsection 4.1]{PowDiag}, it is gotten from the graph of the function $\bm x \mapsto \max_{\bm s\in\bm S} (-\bm s^\top \bm x)$.
The minus sign comes from considering farthest point power diagrams.
After intersecting the epigraph with $\KK^n_{\geq 0}$, the dual is isomorphic to $\conv(\bm S)+\KK^n_{\geq 0}$.
Since Lemma~\ref{lemma:polyPR} gives the polyhedrality of the Voronoi cells in $\PD(\bm S)$, a result of Klee \cite[Corollary~5.14]{Klee:1959} implies that $\conv(\bm S)+\KK^n_{\geq 0}$ is quasi-polyhedral.
The faces of $\PD(\bm S)$ touching $\KK^n_{\geq 0}$ map to the unbounded faces of $\conv(\bm S)+\KK^n_{\geq 0}$.
This leads to the following result.

\begin{theorem}\label{th:delaunay}
  Let $S\subset\cH$ be super-discrete $(r,R)$-system in general position, for $0<r<R\leq\infty$.
  Then, for any $t$ sufficiently large, the Delone complex $\Del(S)$ is isomorphic to the bounded subcomplex of $\conv\smallSetOf{t^{-s}}{s\in S}+\RR^n_{\geq 0}$, which is an unbounded ordinary convex quasi-polyhedron in $\RR^n$.
\end{theorem}

We call a super-discrete $(r,R)$-system $S$ \emph{sufficiently generic} if any two sites whose Voronoi regions intersect are in general position.
In \cite{RRlattice} the authors discuss generic perturbations of rational lattices.
These form our key examples, such as the following.

\begin{example}\label{exmp:perturb}
  We consider the lattice $L=L_2(2,1,1)$ from Example~\ref{exmp:lattice}, which is rational.
  Further, we pick a small rational $\epsilon>0$ to define the lattice $L^\epsilon$ with generators $b_0^\epsilon=(2+2\epsilon,-2-\epsilon,-\epsilon)$ and $b_1^\epsilon=(-1-\epsilon,2+2\epsilon,-1-\epsilon)$.
  The lattice points in $L^\epsilon$ are sufficiently generic but not in general position.
  The origin has eight adjacent Voronoi regions in $\VD(L)$.
  Its Voronoi region is depicted in Figure~\ref{fig:Del}, with and without perturbation.
  Locally, the situation is fully described by nine points in $L$ and their perturbations.
  The simplicial complexes $\Del(L)$ and $\Del(L^\epsilon)$ are three- and two-dimensional, respectively.
\end{example}

Finally, we turn to commutative algebra.
We view the Laurent polynomial ring $\FF[x_1^\pm,\dots,x_n^\pm]$ over some field $\FF$ as an algebra over the polynomial ring $\FF[x_1,\dots,x_n]$.
To be able to make the connection with asymmetric tropical Voronoi diagrams, we now additionally assume that the sites in $S$ have integral coordinates; i.e., $S\subset\cH\cap\ZZ^n$.
Then we obtain the \emph{Laurent monomial module}
\begin{equation}\label{eq:monomial-module}
  M(S) \ = \ \bigl[ \, x_1^{-s}\dots x_n^{-s_n}:s\in S \, \bigr] \enspace,
\end{equation}
which is the submodule of $\FF[x_1^\pm,\dots,x_n^\pm]$ spanned by the monomials $x_1^{-s}\dots x_n^{-s_n}$.
Monomial modules and their relevance to homological properties of commutative rings are the topic of \cite[\textsection 9.2]{CombCommAlg}. 
In our setting the bounded subcomplex of $\conv\smallSetOf{t^{-s}}{s\in S}+\RR^n_{\geq 0}$ is known as the \emph{hull complex} of $M(S)$; see \cite[\textsection 4.4]{CombCommAlg}.
This is known to be isomorphic to the Scarf complex, when we assume genericity; see \cite[Theorem~9.24]{CombCommAlg}.
\begin{corollary}\label{cor:scarf}
  Let $S\subset\cH\cap\ZZ^n$ be a subset of a lattice which is sufficiently generic.
  Then the Delone complex $\Del(S)$ is isomorphic as a simplicial complex to the hull complex of the Laurent monomial module $M(S)$.
\end{corollary}
\begin{proof}
  Theorem~\ref{th:PuiseuxLift} holds for $S$ as the defining halfspaces of any Voronoi region are induced by sites in general position. 
  So the result is a direct consequence of Theorem~\ref{th:delaunay}.
\end{proof}

\begin{figure}[th]\centering
  \begin{tikzpicture}
    \definecolor{tolorange}{RGB}{238,119,51} 
	\definecolor{tolblue}{RGB}{0,119,187}    
	\definecolor{tolgreen}{RGB}{0,153,136}   
	\definecolor{tolpurple}{RGB}{136,34,85}  
    \definecolor{ibmyellow}{RGB}{255, 176, 0}  
    
    
    \draw[tolpurple,opacity = 0.6] (1.732,3) -- (-1.732,3) -- (-3.464,0) -- (-1.732,-3) -- (1.732,-3) -- (3.464,0) -- cycle;
    \draw[tolpurple,opacity = 0.6] (1.732,3) -- (-1.732,-3);
    \draw[tolpurple,opacity = 0.6] (-1.732,3) -- (1.732,-3);
    \draw[tolpurple,opacity = 0.6] (-3.464,0) -- (3.464,0);
    \draw[tolpurple,opacity = 0.6] (0,3) -- (2.598,-1.5) -- (-2.598,-1.5) -- cycle;
    \draw[tolpurple,opacity = 0.6] (0,-3) -- (-2.598,1.5) -- (2.598,1.5) -- cycle;
    
    \fill[color=tolorange, opacity=0.25] (0,1) -- (0.866,0.5) -- (0.866,-0.5) -- (0,-1) -- (-0.866,-0.5) -- (-0.866,0.5) -- cycle;
    
    \draw[color=tolorange, very thick] (0,-2) -- (0,-1) -- (0.866,-0.5) -- (0.866,0.5) -- (1.732,1);
    \draw[color=tolorange, very thick] (1.732,1) -- (0.866,0.5) -- (0,1) -- (-0.866,0.5) -- (-1.732,1);
    \draw[color=tolorange, very thick] (-1.732,1) -- (-0.866,0.5) -- (-0.866,-0.5) -- (0,-1) -- (0,-2);
    
    \node at (0,0) [circle,fill,inner sep=1.5pt,color=tolpurple,draw = tolpurple!50!black, semithick] {};
    \node at (0.866,1.5) [circle,fill,inner sep=1.5pt,color=tolpurple,draw = ibmyellow, semithick] {};
    \node at (1.732,0) [circle,fill,inner sep=1.5pt,color=tolpurple,draw = ibmyellow, semithick] {};
    \node at (-0.866,1.5) [circle,fill,inner sep=1.5pt,color=tolpurple,draw = ibmyellow, semithick] {};
    \node at (0.866,-1.5) [circle,fill,inner sep=1.5pt,color=tolpurple,draw = ibmyellow, semithick] {};
    \node at (-1.732,0) [circle,fill,inner sep=1.5pt,color=tolpurple,draw = ibmyellow, semithick] {};
    \node at (-0.866,-1.5) [circle,fill,inner sep=1.5pt,color=tolpurple,draw = ibmyellow, semithick] {};
    \node at (1.732,3) [circle,fill,inner sep=1.5pt,color=tolpurple,draw = ibmyellow, semithick] {};
    \node at (2.598,1.5) [circle,fill,inner sep=1.5pt,color=tolpurple,draw = ibmyellow, semithick] {};
    \node at (3.464,0) [circle,fill,inner sep=1.5pt,color=tolpurple,draw = ibmyellow, semithick] {};
    \node at (0,3) [circle,fill,inner sep=1.5pt,color=tolpurple,draw = ibmyellow, semithick] {};
    \node at (2.598,-1.5) [circle,fill,inner sep=1.5pt,color=tolpurple,draw = ibmyellow, semithick] {};
    \node at (-1.732,3) [circle,fill,inner sep=1.5pt,color=tolpurple,draw = ibmyellow, semithick] {};
    \node at (1.732,-3) [circle,fill,inner sep=1.5pt,color=tolpurple,draw = ibmyellow, semithick] {};
    \node at (-2.598,1.5) [circle,fill,inner sep=1.5pt,color=tolpurple,draw = ibmyellow, semithick] {};
    \node at (0,-3) [circle,fill,inner sep=1.5pt,color=tolpurple,draw = ibmyellow, semithick] {};
    \node at (-3.464,0) [circle,fill,inner sep=1.5pt,color=tolpurple,draw = ibmyellow, semithick] {};
    \node at (-2.598,-1.5) [circle,fill,inner sep=1.5pt,color=tolpurple,draw = ibmyellow, semithick] {};
    \node at (-1.732,-3) [circle,fill,inner sep=1.5pt,color=tolpurple,draw = ibmyellow, semithick] {};
    
    \node at (0,1) [circle,fill,inner sep=1.5pt,color=tolblue,draw = white, semithick] {};
    \node at (0.866,-0.5) [circle,fill,inner sep=1.5pt,color=tolblue,draw = white, semithick] {};
    \node at (-0.866,-0.5) [circle,fill,inner sep=1.5pt,color=tolblue,draw = white, semithick] {};
    \node at (1.732,1) [circle,fill,inner sep=1.5pt,color=tolblue,draw = white, semithick] {};
    \node at (-1.732,1) [circle,fill,inner sep=1.5pt,color=tolblue,draw = white, semithick] {};
    \node at (0,-2) [circle,fill,inner sep=1.5pt,color=tolblue,draw = white, semithick] {};

\end{tikzpicture} 
  \caption{Voronoi region $\VR_{A_2}(\0)$, whose six tropical vertices lie in the boundary of $-2\triangle$;  and Delone complex of $A_2$}
  \label{fig:DelA2}
\end{figure}
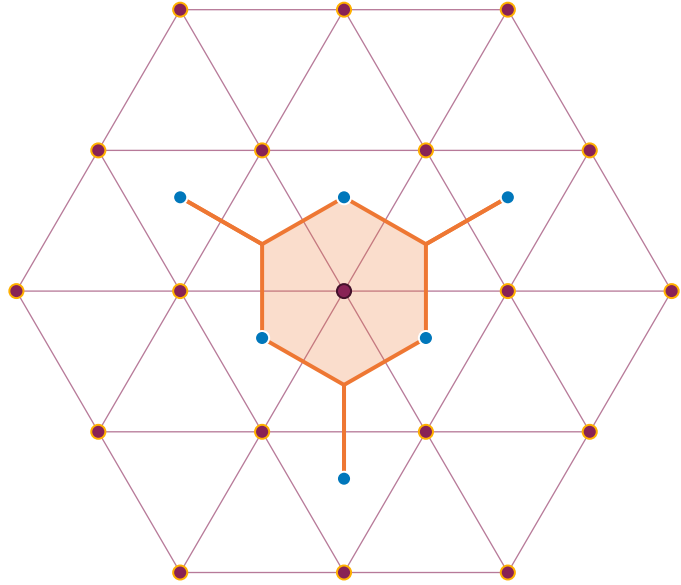

Connections between monomial ideals/modules and tropical geometry have been discussed in \cite{DevelinYu07}, \cite{ManjunathSturmfels:2013} and elsewhere.
In \cite{MonCon} it was shown that this connection can be exploited in multicriteria optimization \cite{Ehrgott:2005}.
Seeing the exponents of monomials as points in $\RR^n$ leads to geometric objects called \emph{staircases} in \cite[\textsection 3]{CombCommAlg} which are tropical cones in the view of \cite{MonCon}.
In the homogeneous case, Amini and Manjunath link these cones to tropical Voronoi diagrams in \cite[\textsection 4.3]{RRlattice}.
More exactly, they look at the epigraph of the function $\frac{1}{n}\min_{s\in S}d_\triangle(\cdot,s)$, which encodes the distance to the closest site.
With the terminology from \cite{MonCon}, the epigraph is the tropical monomial cone generated by $-S$ and its boundary projects onto $\VD(S)$.

We close this paper with the analysis of a classical example, which also occurs in \cite[Corollary~28]{ManjunathSturmfels:2013}.
\begin{example}
  The root lattice $A_{n-1}$ is generated by the vectors $e_i-e_j$, for $i,j\in[n]$; it lies in $\cH$.
  The asymmetric tropical Voronoi diagrams of $A_{n-1}$ and its sublattices are studied in \cite[\textsection\textsection 4.2, 4.3]{RRlattice}.

  The tropical vertices of the Voronoi region $\VR_{A_{n-1}}(\0)$ are points with integral coordinates in the dilated simplex $-(n-1)\triangle$.
  The integer points in an intersection $\VR_{A_{n-1}}(x)\cap\VR_{A_{n-1}}(y)$ arise from the intersection of two translated copies of $-(n-1)\triangle$.
  If non-empty, that intersection is of the form $z-r\triangle$ for some $z\in\ZZ^n/\ZZ\1$, so $\dim(\VR_{A_{n-1}}(x)\cap\VR_{A_{n-1}}(y))=r$.
  We have $r=n-2$ if and only if $y=x+e_i-e_j$ for some distinct indices $i,j\in[n]$.
  This characterizes the dual graph of the Voronoi diagram $\VD(A_{n-1})$, from which we can derive the Delone complex.
  
  We conclude that $\Del(A_{n-1})$ is isomorphic to the standard triangulation of an apartment in the Bruhat--Tits building of the group $\SL_n(\FF)$ over a field $\FF$ with a discrete valuation; see \cite[p.753f]{ManjunathSturmfels:2013} and \cite[Observation~10.82]{ETC} for the connection to tropical convexity.
  In this way $\Del(A_{n-1})$ may also be seen as a geometric realization of the affine Coxeter group of type $\widetilde{A}_{n-1}$.
  An example for $\FF$ is given by the ordinary (dual) Laurent series with complex coefficients, which forms a subfield of the field of generalized dual Puiseux $\hahnseries{\CC}{t}^*$.

  The computation above also shows that $A_{n-1}$ is not sufficiently generic.
  Nonetheless, its Delone complex is pure of dimension $n-1$.
\end{example}

\printbibliography

@book{ETC,
  author = {Joswig, Michael},
  title = {Essentials of tropical combinatorics},
  publisher = {American Mathematical Society},
  address = {Providence, RI},
  series = {Graduate Studies in Mathematics},
  volume = {219},
  year = {2021},
}

@article {TropBis,
  author = {Criado, Francisco and Joswig, Michael and Santos, Francisco},
  title =  {Tropical bisectors and {V}oronoi diagrams},
  journal = {Found. Comput. Math.},
  year = 2021,
  note = {Published online: 07 September 2021},
  arxiv =   {1906.10950},
  doi = {10.1007/s10208-021-09538-4}
}

@article {RRlattice,
    AUTHOR = {Amini, Omid and Manjunath, Madhusudan},
     TITLE = {Riemann-{R}och for sub-lattices of the root lattice {$A_n$}},
   JOURNAL = {Electron. J. Combin.},
  FJOURNAL = {Electronic Journal of Combinatorics},
    VOLUME = {17},
      YEAR = {2010},
    NUMBER = {1},
     PAGES = {Research Paper 124, 50},
   MRCLASS = {06B99 (05C50 05E99 52C07)},
  MRNUMBER = {2729373},
MRREVIEWER = {Tatiana Smirnova-Nagnibeda},
       _URL =
              {http://www.combinatorics.org/Volume_17/Abstracts/v17i1r124.html},
}

@article {Lin-Yoshida:2018,
    AUTHOR = {Lin, Bo and Yoshida, Ruriko},
     TITLE = {Tropical {F}ermat-{W}eber points},
   JOURNAL = {SIAM J. Discrete Math.},
  FJOURNAL = {SIAM Journal on Discrete Mathematics},
    VOLUME = {32},
      YEAR = {2018},
    NUMBER = {2},
     PAGES = {1229--1245},
      ISSN = {0895-4801},
   MRCLASS = {13P25 (14T90 52B11 92B15)},
  MRNUMBER = {3810501},
MRREVIEWER = {Mateusz Micha\l ek},
       DOI = {10.1137/16M1071122},
       URL = {https://doi.org/10.1137/16M1071122},
}

@article {PowDiag,
    AUTHOR = {Aurenhammer, Franz},
     TITLE = {Power diagrams: properties, algorithms and applications},
   JOURNAL = {SIAM J. Comput.},
  FJOURNAL = {SIAM Journal on Computing},
    VOLUME = {16},
      YEAR = {1987},
    NUMBER = {1},
     PAGES = {78--96},
      ISSN = {0097-5397},
   MRCLASS = {68U05 (11H50 52-04 90C10)},
  MRNUMBER = {873251},
       DOI = {10.1137/0216006},
       URL = {https://doi.org/10.1137/0216006},
}

@article {MonCon,
    AUTHOR = {Joswig, Michael and Loho, Georg},
     TITLE = {Monomial tropical cones for multicriteria optimization},
   JOURNAL = {SIAM J. Discrete Math.},
  FJOURNAL = {SIAM Journal on Discrete Mathematics},
    VOLUME = {34},
      YEAR = {2020},
    NUMBER = {2},
     PAGES = {1172--1191},
      ISSN = {0895-4801},
   MRCLASS = {90C29 (13D02 14T15 90C24)},
  MRNUMBER = {4101364},
MRREVIEWER = {Haohao Li},
       DOI = {10.1137/17M1153066},
       URL = {https://doi.org/10.1137/17M1153066},
}

@book {CombCommAlg,
    AUTHOR = {Miller, Ezra and Sturmfels, Bernd},
     TITLE = {Combinatorial commutative algebra},
    SERIES = {Graduate Texts in Mathematics},
    VOLUME = {227},
 PUBLISHER = {Springer-Verlag, New York},
      YEAR = {2005},
     _PAGES = {xiv+417},
      ISBN = {0-387-22356-8},
   MRCLASS = {13-01 (05-01 05E99 13D02 14M15 14M25)},
  MRNUMBER = {2110098},
MRREVIEWER = {Joseph Gubeladze},
}

@article {LapLatticeGraph,
    AUTHOR = {Manjunath, Madhusudan},
     TITLE = {The {L}aplacian lattice of a graph under a simplicial distance function},
   JOURNAL = {European J. Combin.},
  FJOURNAL = {European Journal of Combinatorics},
    VOLUME = {34},
      YEAR = {2013},
    NUMBER = {6},
     PAGES = {1051--1070},
      ISSN = {0195-6698},
   MRCLASS = {52B20 (05C50 06A07)},
  MRNUMBER = {3037988},
MRREVIEWER = {Stephen J. Young},
       DOI = {10.1016/j.ejc.2013.01.010},
       _URL = {https://doi.org/10.1016/j.ejc.2013.01.010},
}

@article {MartiniSwanepoel:2004,
    AUTHOR = {Martini, Horst and Swanepoel, Konrad J.},
     TITLE = {The geometry of {M}inkowski spaces---a survey. {II}},
   JOURNAL = {Expo. Math.},
  FJOURNAL = {Expositiones Mathematicae},
    VOLUME = {22},
      YEAR = {2004},
    NUMBER = {2},
     PAGES = {93--144},
      _ISSN = {0723-0869},
   MRCLASS = {46B20 (52A21)},
  MRNUMBER = {2056652},
MRREVIEWER = {Apostolos A. Giannopoulos},
       DOI = {10.1016/S0723-0869(04)80009-4},
       _URL = {https://doi.org/10.1016/S0723-0869(04)80009-4},
}

@book {AurenhammerKleinLee:2013,
    AUTHOR = {Aurenhammer, Franz and Klein, Rolf and Lee, Der-Tsai},
     TITLE = {Voronoi diagrams and {D}elaunay triangulations},
 PUBLISHER = {World Scientific Publishing Co. Pte. Ltd., Hackensack, NJ},
      YEAR = {2013},
     __PAGES = {viii+337},
      _ISBN = {978-981-4447-63-8},
   MRCLASS = {52-02 (52C17 52C20)},
  MRNUMBER = {3186045},
MRREVIEWER = {Mathieu Dutour Sikiri\'{c}},
       DOI = {10.1142/8685},
       _URL = {https://doi.org/10.1142/8685},
}

@article{Markwig:2010,
    AUTHOR = {Markwig, Thomas},
     TITLE = {A field of generalised {P}uiseux series for tropical geometry},
   JOURNAL = {Rend. Semin. Mat. Univ. Politec. Torino},
  FJOURNAL = {Rendiconti del Seminario Matematico. Universit\`a e Politecnico Torino},
    VOLUME = {68},
      YEAR = {2010},
    NUMBER = {1},
     PAGES = {79--92},
      ISSN = {0373-1243},
   MRCLASS = {14T05},
  MRNUMBER = {2759691 (2012e:14126)},
}

@book {AffDiffGeomHypsurf,
    AUTHOR = {Li, An-Min and Simon, Udo and Zhao, Guosong and Hu, Zejun},
     TITLE = {Global affine differential geometry of hypersurfaces},
    SERIES = {De Gruyter Expositions in Mathematics},
    VOLUME = {11},
   EDITION = {extended},
 PUBLISHER = {De Gruyter, Berlin},
      YEAR = {2015},
     _PAGES = {x+365},
      ISBN = {978-3-11-026667-2; 978-3-11-039090-2},
   MRCLASS = {53A15 (53-02 53C40)},
  MRNUMBER = {3382197},
MRREVIEWER = {Francisco Mil\'{a}n},
       DOI = {10.1515/9783110268898},
       URL = {https://doi.org/10.1515/9783110268898},
}

@book {BoissonnatYvinec:1998,
    AUTHOR = {Boissonnat, Jean-Daniel and Yvinec, Mariette},
     TITLE = {Algorithmic geometry},
      NOTE = {Translated from the 1995 French original by Herv\'{e} Br\"{o}nnimann},
 PUBLISHER = {Cambridge University Press, Cambridge},
      YEAR = {1998},
     _PAGES = {xxii+519},
      ISBN = {0-521-56529-4},
   MRCLASS = {68U05 (52B55 68-02)},
  MRNUMBER = {1631175},
MRREVIEWER = {Hans-Dietrich Hecker},
       DOI = {10.1017/CBO9781139172998},
       URL = {https://doi.org/10.1017/CBO9781139172998},
}

@book {BasuPollackRoy:2006,
    AUTHOR = {Basu, Saugata and Pollack, Richard and Roy, Marie-Fran\c{c}oise},
     TITLE = {Algorithms in real algebraic geometry},
    SERIES = {Algorithms and Computation in Mathematics},
    VOLUME = {10},
   EDITION = {Second},
 PUBLISHER = {Springer-Verlag, Berlin},
      YEAR = {2006},
     __PAGES = {x+662},
      _ISBN = {978-3-540-33098-1; 3-540-33098-4},
   MRCLASS = {14P10 (03C10 52C45 68Q25 68W30)},
  MRNUMBER = {2248869},
}

@article {ManjunathSturmfels:2013,
     AUTHOR = {Manjunath, Madhusudan and Sturmfels, Bernd},
      TITLE = {Monomials, binomials and {R}iemann-{R}och},
    JOURNAL = {J. Algebraic Combin.},
   FJOURNAL = {Journal of Algebraic Combinatorics. An International Journal},
     VOLUME = {37},
       YEAR = {2013},
     NUMBER = {4},
      PAGES = {737--756},
       ISSN = {0925-9899},
    MRCLASS = {05Exx (05C50)},
   MRNUMBER = {3047017},
MRREVIEWER = {Jan \'{S}liwa},
        DOI = {10.1007/s10801-012-0386-9},
        _URL = {https://doi.org/10.1007/s10801-012-0386-9}
}

@PhdThesis{Voigt:2008,
  author = 	 {Voigt, Ina Kirsten},
  title = 	 {Voronoizellen diskreter {P}unktmengen},
  school = 	 {TU Dortmund},
  year = 	 2008,
 MSC2010 = {52-02 52A37 52B99},
 Zbl = {1332.52001}
}

@book {Gruber:2007,
    AUTHOR = {Gruber, Peter M.},
     TITLE = {Convex and discrete geometry},
    SERIES = {Grundlehren der mathematischen Wissenschaften},
    VOLUME = {336},
 PUBLISHER = {Springer, Berlin},
      YEAR = {2007},
     _PAGES = {xiv+578},
      ISBN = {978-3-540-71132-2},
   MRCLASS = {52-02 (11-02 46B20 49-02 49J53 90C25)},
  MRNUMBER = {2335496},
MRREVIEWER = {Aleksandr Koldobsky},
}

@book {HerzogHibi:2011,
    AUTHOR = {Herzog, J\"{u}rgen and Hibi, Takayuki},
     TITLE = {Monomial ideals},
    SERIES = {Graduate Texts in Mathematics},
    VOLUME = {260},
 PUBLISHER = {Springer-Verlag London, Ltd., London},
      YEAR = {2011},
     _PAGES = {xvi+305},
      ISBN = {978-0-85729-105-9},
   MRCLASS = {13D02 (05E40 13D40 13F55 13P10)},
  MRNUMBER = {2724673},
MRREVIEWER = {Rahim Zaare-Nahandi},
       DOI = {10.1007/978-0-85729-106-6},
       URL = {https://doi.org/10.1007/978-0-85729-106-6},
}

@article {BSTY1998,
    AUTHOR = {Boissonnat, J.-D. and Sharir, M. and Tagansky, B. and Yvinec,
              M.},
     TITLE = {Voronoi diagrams in higher dimensions under certain polyhedral
              distance functions},
   JOURNAL = {Discrete Comput. Geom.},
  FJOURNAL = {Discrete \& Computational Geometry. An International Journal
              of Mathematics and Computer Science},
    VOLUME = {19},
      YEAR = {1998},
    NUMBER = {4},
     PAGES = {485--519},
      ISSN = {0179-5376},
   MRCLASS = {52B55 (68Q25 68U05)},
  MRNUMBER = {1620060},
MRREVIEWER = {Rade \v{Z}ivaljevi\'{c}},
       DOI = {10.1007/PL00009366},
       URL = {https://doi.org/10.1007/PL00009366},
}

@book {HardyWright:2008,
    AUTHOR = {Hardy, G. H. and Wright, E. M.},
     TITLE = {An introduction to the theory of numbers},
   EDITION = {Sixth},
      NOTE = {Revised by D. R. Heath-Brown and J. H. Silverman,
              With a foreword by Andrew Wiles},
 PUBLISHER = {Oxford University Press, Oxford},
      YEAR = {2008},
     __PAGES = {xxii+621},
      ISBN = {978-0-19-921986-5},
   MRCLASS = {11-01},
  MRNUMBER = {2445243},
}

@ARTICLE{CohenGaubertQuadrat04,
    AUTHOR = {Cohen, Guy and Gaubert, St{\'e}phane and Quadrat, Jean-Pierre},
     TITLE = {Duality and separation theorems in idempotent semimodules},
      NOTE = {Tenth Conference of the International Linear Algebra Society},
   JOURNAL = {Linear Algebra Appl.},
  FJOURNAL = {Linear Algebra and its Applications},
    VOLUME = {379},
      YEAR = {2004},
     PAGES = {395--422},
      ISSN = {0024-3795},
     CODEN = {LAAPAW},
   MRCLASS = {46A20 (06F07 46A55 93C65)},
  MRNUMBER = {2039751 (2005e:46007)},
MRREVIEWER = {Bart De Schutter},
       DOI = {10.1016/j.laa.2003.08.010},
}

@book {Tropical+Book,
    AUTHOR = {Maclagan, Diane and Sturmfels, Bernd},
     TITLE = {Introduction to tropical geometry},
    SERIES = {Graduate Studies in Mathematics},
    VOLUME = {161},
 PUBLISHER = {American Mathematical Society, Providence, RI},
      YEAR = {2015},
     PAGES = {xii+363},
      ISBN = {978-0-8218-5198-2},
   MRCLASS = {14T05 (05B35 15A80 52B70)},
  MRNUMBER = {3287221},
}

@article{ABGJ:2018,
  author = {Allamigeon, Xavier and Benchimol, Pascal and Gaubert, St\'ephane and Joswig, Michael},
  title =  {Log-barrier interior point methods are not strongly polynomial},
  journal = {SIAM J. Appl. Algebra Geom.},
  year = {2018},
  volume = 2,
  number = 1,
  pages = {140-178},
  doi = {10.1137/17M1142132},
  arxiv = {1708.01544}
}

@Misc{ComaneciJoswig:2205.00036,
  author = {Comăneci, Andrei and Joswig, Michael},
  title =  {Tropical medians by transportation},
  year =   {2022},
  note =   {Preprint \arXiv{2205.00036}}
}

@book {Ehrgott:2005,
    AUTHOR = {Ehrgott, Matthias},
     TITLE = {Multicriteria optimization},
   EDITION = {Second},
 PUBLISHER = {Springer-Verlag, Berlin},
      YEAR = {2005},
      ISBN = {3-540-21398-8},
   MRCLASS = {90-01 (90C29)},
  MRNUMBER = {2143243},
}

@article {Klee:1959,
    AUTHOR = {Klee, Victor},
     TITLE = {Some characterizations of convex polyhedra},
   JOURNAL = {Acta Math.},
  FJOURNAL = {Acta Mathematica},
    VOLUME = {102},
      YEAR = {1959},
     PAGES = {79--107},
      ISSN = {0001-5962},
   MRCLASS = {52.00},
  MRNUMBER = {105651},
MRREVIEWER = {H. Mirkil},
       DOI = {10.1007/BF02559569},
       URL = {https://doi.org/10.1007/BF02559569},
}

@article {Edelsbrunner+Seidel:1986,
    AUTHOR = {Edelsbrunner, Herbert and Seidel, Raimund},
     TITLE = {Vorono\u{\i} diagrams and arrangements},
   JOURNAL = {Discrete Comput. Geom.},
  FJOURNAL = {Discrete \& Computational Geometry. An International Journal of Mathematics and Computer Science},
    VOLUME = {1},
      YEAR = {1986},
    NUMBER = {1},
     PAGES = {25--44},
      ISSN = {0179-5376},
   MRCLASS = {68Q25 (52A37 68U05)},
  MRNUMBER = {824106},
MRREVIEWER = {D. T. Lee},
       DOI = {10.1007/BF02187681},
       URL = {https://doi.org/10.1007/BF02187681},
}

@article {Gaubert+Katz:2011,
     AUTHOR = {Gaubert, St\'{e}phane and Katz, Ricardo D.},
      TITLE = {Minimal half-spaces and external representation of tropical polyhedra},
    JOURNAL = {J. Algebraic Combin.},
   FJOURNAL = {Journal of Algebraic Combinatorics. An International Journal},
     VOLUME = {33},
       YEAR = {2011},
     NUMBER = {3},
      PAGES = {325--348},
       ISSN = {0925-9899},
    MRCLASS = {52A20 (14T05 52B55)},
   MRNUMBER = {2772536},
MRREVIEWER = {Aris Daniilidis},
        DOI = {10.1007/s10801-010-0246-4},
        URL = {https://doi.org/10.1007/s10801-010-0246-4},
}

@ARTICLE{DevelinYu07,
  author = {Develin, Mike and Yu, Josephine},
  title = {Tropical polytopes and cellular resolutions},
  journal = {Experiment. Math.},
  year = {2007},
  volume = {16},
  pages = {277--291},
  number = {3},
  fjournal = {Experimental Mathematics},
  issn = {1058-6458},
  mrclass = {52A30 (13D02)},
  mrnumber = {MR2367318}
}

\end{document}